\documentclass[a4paper, twoside]{article} 

\usepackage[
	top = 1.15 in, 
	bottom = 1.25 in,
	left = 1.15 in, 
	right = 1.15 in,
	includehead]{geometry}
\usepackage{scrextend}
\changefontsizes{11 pt}

\usepackage[all]{nowidow}

\usepackage[utf8x]{inputenc}

\usepackage{amscd,amssymb,amsmath}
\usepackage{amsrefs}

\usepackage{mathptmx}
\usepackage[T1]{fontenc}

\usepackage{multicol}

\usepackage[bb=ams, cal=euler, scr=rsfso , frak=euler]{mathalfa}

\usepackage{url}
\usepackage[normalem]{ulem}

\BibSpec{article}{%
    +{}  {\PrintAuthors}                {author}
    +{,} { \textit}                     {title}
    +{.} { }                            {part}
    +{:} { \textit}                     {subtitle}
    +{,} { \PrintContributions}         {contribution}
    +{.} { \PrintPartials}              {partial}
    +{,} { }                            {journal}
    +{}  { \textbf}                     {volume}
    +{}  { \PrintDatePV}                {date}
    +{,} { \issuetext}                  {number}
    +{,} { \eprintpages}                {pages}
    +{,} { }                            {status}
    +{,} { \url}                        {url}    % <---- ADDED
    +{,} { \PrintDOI}                   {doi}
    +{,} { available at \eprint}        {eprint}
    +{}  { \parenthesize}               {language}
    +{}  { \PrintTranslation}           {translation}
    +{;} { \PrintReprint}               {reprint}
    +{.} { }                            {note}
    +{.} {}                             {transition}
    +{}  {\SentenceSpace \PrintReviews} {review}
}

\usepackage{graphicx}
\usepackage[british]{babel}
\usepackage{caption}
\usepackage{mathdots}
\usepackage{mathtools}

\usepackage{pgf,tikz}
\usepackage{tikz-cd}
\usetikzlibrary{decorations.text}
\usetikzlibrary{shapes.arrows}
\usetikzlibrary{calc}
\usetikzlibrary{matrix,arrows}
\usepackage{stackrel}
\usepackage[shortlabels]{enumitem} %Para listas estilo Weibel
\usepackage{setspace} %Para espaciado de renglones
\spacing{1.15}
\usepackage{etoolbox}
\usetikzlibrary{trees}

\usepackage{verbatim}

\usepackage{graphicx}
\usepackage{subfigure}
\usepackage{makeidx}
\usepackage{array}
\usepackage{cancel}
\usepackage{polynom}
\usepackage[pdf,all]{xy}
\xyoption{line}
\CompileMatrices

\newcommand{\hem}{\hspace{0.5 em}} 

\usepackage{pgf,tikz}
\usepackage{tikz-cd}
\usetikzlibrary{calc}
\usetikzlibrary{matrix,arrows}
\usepackage{stackrel}
\usepackage[shortlabels]{enumitem} %Para listas estilo Weibel
\usepackage{setspace} %Para espaciado de renglones
\spacing{1.15}
\usepackage{etoolbox}
\usetikzlibrary{trees}

\usepackage{enumitem}
\setlist[enumerate]{label= (\arabic*)}

\patchcmd{\section}{\normalfont}{\normalfont\large}{}{}

\usepackage{wrapfig, framed, caption}

\definecolor{niceblue}{rgb}{0.03, 0.27, 0.49}

\usepackage{fancyhdr}
\usepackage{letltxmacro}

\AtBeginDocument{%
    \LetLtxMacro\refa{\ref}%
    \DeclareRobustCommand{\ref}[2][]{(\refa#1{#2})}%
}
\renewenvironment{abstract}{%
\small\begin{center}
\begin{minipage}{.9\textwidth}
}
{\par\noindent\end{minipage}\end{center}\vspace{3 em}}
\makeatletter
\renewcommand\@maketitle{%
\hfill
\begin{center}\begin{minipage}{0.9 	\textwidth}
\centering
\vskip 2em
\let\footnote\thanks 
{\LARGE \@title \par }
\vspace{-.5 em}
\vskip 1 em
{\large \@author \par}
\vspace{3.5 em}

\end{minipage}\end{center}
\par
}
\makeatother
\definecolor{newcol}{rgb}{0.0, 0, 0}
\DeclareTextFontCommand{\new}{\color{newcol}\em}

\usepackage[pdftex, colorlinks,bookmarks 
= true,bookmarksnumbered = 
true]{hyperref}

\usepackage{sectsty}

\chapterfont{\color{newcol}}  
\sectionfont{\color{newcol}}  
\subsectionfont{\color{newcol}}  
\mathchardef\hy="2D % Define a "math hyphen"

\usepackage{cutwin}

\makeatletter
\AtBeginDocument{%
  \let\[\@undefined
  
\DeclareRobustCommand{\[}{\begin{equation}}%
  \let\]\@undefined
  
\DeclareRobustCommand{\]}{\end{equation}}%
}
\makeatother 
\mathtoolsset{showonlyrefs,showmanualtags}

\usepackage{amsthm}
\usepackage{thmtools}
\newtheoremstyle{mytheorem}
  {\topsep}   % ABOVESPACE
  {\topsep}   % BELOWSPACE
  {\itshape}  % BODYFONT
  {0pt}       % INDENT (empty value is the same as 0pt)
  {\bfseries\color{newcol}} % HEADFONT
  {\color{newcol}}         % HEADPUNCT
  {5pt plus 1pt minus 1pt} % HEADSPACE
  {}          % CUSTOM-HEAD-SPEC
  
\theoremstyle{mytheorem}
\newtheorem{theorem}{Theorem}[section]
\newtheorem{prop}[theorem]{Proposition} 
\newtheorem{lemma}[theorem]{Lemma} 
\newtheorem*{conj*}{Conjecture}
\newtheorem{cor}[theorem]{Corollary} 

\newtheoremstyle{mytheorem2}
  {\topsep}   % ABOVESPACE
  {\topsep}   % BELOWSPACE
  {\itshape}  % BODYFONT
  {0pt}       % INDENT (empty value is the same as 0pt)
  {\bfseries\color{newcol}} % HEADFONT
  {\color{newcol}{.}}         % HEADPUNCT
  {5pt plus 1pt minus 1pt} % HEADSPACE
  {}          % CUSTOM-HEAD-SPEC
  
\theoremstyle{mytheorem2}
\newtheorem*{theorem*}{Theorem}
\newtheorem*{lemma*}{Lemma}
\newtheorem*{corollary*}{Corollary}
\newtheorem*{conjecture*}{Conjecture}

\newtheoremstyle{mydefinition}
  {\topsep}   % ABOVESPACE
  {\topsep}   % BELOWSPACE
  {}  % BODYFONT
  {0pt}       % INDENT (empty value is the same as 0pt)
  {\bfseries\color{newcol}} % HEADFONT
  {\color{newcol}}         % HEADPUNCT
  {5pt plus 1pt minus 1pt} % HEADSPACE
  {}          % CUSTOM-HEAD-SPEC

\theoremstyle{mydefinition}

\newtheorem{definition}[theorem]{Definition}
\newtheorem{remark}[theorem]{Remark}

\newtheoremstyle{mydefinition2}
  {\topsep}   % ABOVESPACE
  {\topsep}   % BELOWSPACE
  {}  % BODYFONT
  {0pt}       % INDENT (empty value is the same as 0pt)
  {\bfseries\color{newcol}} % HEADFONT
  {\color{newcol}{.}}         % HEADPUNCT
  {5pt plus 1pt minus 1pt} % HEADSPACE
  {}          % CUSTOM-HEAD-SPEC
  
\theoremstyle{mydefinition2}
\newtheorem*{definition*}{Definition}
\newtheorem*{remark*}{Remark}
\newtheorem*{obs*}{Observation}
\newtheorem*{example*}{Example}
\newtheorem*{que*}{Question}

\usepackage{graphicx}

\newcommand{\antishriek}{\text{\normalfont{\raisebox{\depth}{\textexclamdown}}}}

\newcommand{\A}{\mathsf{Ass}}
\newcommand{\C}{\mathsf{Com}}
\renewcommand{\L}{\mathsf{Lie}}
\newcommand{\Lie}{\mathsf{Lie}}
\newcommand{\Perm}{\mathsf{Perm}}
\newcommand{\preLie}{\mathsf{PreLie}}

\newcommand{\PP}{\mathcal{P}}
\newcommand{\QQ}{\mathcal{Q}}
\newcommand{\OO}{\mathcal{O}}

\newcommand{\KK}{\mathcal{K}}

\newcommand{\YY}{\mathcal{Y}}
\newcommand{\XX}{\mathcal{X}}

    \usepackage{relsize}

\newcommand{\Cog}{\mathsf{Cog}}
\newcommand{\Def}{\mathrm{Def}}

\newcommand\Tpoly[1]{\operatorname{Poly}^*(#1)}

\newenvironment{titemize}{
\begin{itemize}
  \setlength{\itemsep}{0pt}
  \setlength{\parskip}{0pt}
}{\end{itemize}}

\newcommand{\imor}{\interleave\kern-.45em\longrightarrow}
\newcommand{\SMod}{{}_\Sigma\mathsf{dgMod}}

\newcommand{\Mod}{\mathsf{Mod}}

\newcommand{\Cxs}{{}_\kk\mathsf{Ch}}

\newcommand{\Der}{\operatorname{Der}}

\newcommand{\g}{\mathfrak{g}}

\newcommand{\HH}{\mathscr{H}}
\newcommand{\HHC}{\mathrm{HH}}

\newcommand{\HKR}{\operatorname{HKR}}
\newcommand{\cof}{\mathit{Q}}
\newcommand{\FF}{\mathcal{F}}
\newcommand{\Br}{\mathsf{Br}}

\newenvironment{tenumerate}{
\begin{enumerate}
  \setlength{\itemsep}{0pt}
  \setlength{\parskip}{0pt}
}{\end{enumerate}}

\usepackage{xargs}                      
\usepackage[colorinlistoftodos,prependcaption,textsize=small]{todonotes}
\newcommandx{\unsure}[2][1=]{\todo[linecolor=blue,backgroundcolor=blue!25!white,bordercolor=blue,#1]{#2}}
\newcommandx{\change}[2][1=]{\todo[linecolor=blue,backgroundcolor=blue!25,bordercolor=blue,#1]{#2}}
\newcommandx{\info}[2][1=]{\todo[linecolor=OliveGreen,backgroundcolor=OliveGreen!25,bordercolor=OliveGreen,#1]{#2}}
\newcommandx{\improvement}[2][1=]{\todo[linecolor=Plum,backgroundcolor=Plum!25,bordercolor=Plum,#1]{#2}}
\newcommandx{\thiswillnotshow}[2][1=]{\todo[disable,#1]{#2}}

\definecolor{col1}{rgb}{0.45, 0.66, 0.76}
\definecolor{col2}{rgb}{0.69, 0.88, 0.9}

\definecolor{comment}{rgb}{0.0, 0.0, 0.61}

\newcommand\id{\mathrm{id}}

\newcommand{\Alg}{\mathsf{Alg}}

\renewcommand{\tt}{\otimes}

\newcommand{\NN}{\mathbb N}
\newcommand{\ZZ}{\mathbb Z}

\newcommand{\kk}{\Bbbk}

\renewcommand{\Bar}{\operatorname{B}} %Decide on good font

\newcommand{\Ext}{\operatorname{Ext}}
\newcommand{\Tor}{\operatorname{Tor}}

\usepackage[new]{old-arrows} %% LONG HOOK ARROW!

\newcommand{\Addresses}{{% additional braces for segregating \footnotesize
  \bigskip
  \footnotesize
  \textsc{IMAG, Univ. Montpellier, CNRS, Montpellier, France}\par\nopagebreak
  \textit{E-mail address:} \texttt{ricardo.campos@umontpellier.fr}
  }
  {\medskip\par\footnotesize
  \textsc{School of Mathematics, Trinity College, Dublin 2, Ireland}\par\nopagebreak
  \textit{E-mail address:} \texttt{pedro@maths.tcd.ie}
  }}
  
\usepackage{titletoc}

\titlecontents{chapter}
[0.2em] %
{\bigskip}
{\makebox[2em][r]{\thecontentslabel.}\hspace{0.333em}}%\thecontentslabel
{\hspace*{-2em}}
{\hfill\contentspage}[\smallskip]

\titlecontents{section}% <section>
[0.2 em]% <left>
{\small}% <above-code>
{\thecontentslabel.\hspace{3pt}}%<numbered-entry-format>; you could also 
{}% <numberless-entry-format>
{\enspace\titlerule*[0.5pc]{.}\contentspage}%<filler-page-format>
\titlecontents*{subsection}% <section>
[1.2 em]% <left>
{\footnotesize}% <above-code>
{\thecontentslabel. \hspace{3pt}}% <numbered-entry-format>; you could also 
{}% <numberless-entry-format>
{}% <filler-page-format>
[ --- \ ]% <separator>
[]% <end>
\hypersetup{colorlinks,
	linkcolor={niceblue},
	citecolor={niceblue},
	urlcolor={niceblue}}  

\DeclareMathVersion{normal2}

\usepackage{lipsum}
\makeindex

\title{\setstretch{0.85}{\textbf{Differential forms on smooth\\ operadic algebras}}}
\author{\textsc{Ricardo Campos \\ Pedro Tamaroff}}

\date{}
\begin{document}
\maketitle

\begin{abstract}
The classical Hochschild--Kostant--Rosenberg
(HKR) theorem computes the Hochschild homology
and cohomology of
smooth commutative algebras. In this paper, we generalise
this result to other kinds of algebraic 
structures. Our main insight is that producing HKR
isomorphisms for other types of algebras 
is directly related to computing quasi-free
resolutions in the category of left modules over an operad;
we establish that an HKR-type result follows
as soon as this resolution is diagonally pure.

As examples we obtain a permutative and a pre-Lie 
HKR theorem for smooth commutative and smooth brace
algebras, respectively. We also prove an HKR theorem
for operads obtained from a filtered distributive
law, which recovers, in particular,
all the aspects of the classical HKR theorem.
Finally, we show that this property is Koszul dual to 
the operadic PBW property defined by V. Dotsenko 
and the second author.

\bigskip


\end{abstract}
\thispagestyle{empty}

\section{Introduction}\label{sec:intro}

%Historic intro

Hochschild homology is a classical homology theory for associative algebras\cite{hochschild1945cohomology} dating back to 1945.
Originally conceived by Hochschild to obtain a cohomological 
proof of Wedderburn's theorem~\cite{Kadison1995},
this cohomology theory plays nowadays important roles in 
representation theory~\cite{Bustamante2016}, 
deformation 
theory~\cite{ginot2016deformation,Gerstenhaber1988,
DefTheoryNotes}, derived geometry~\cite{toen_vezzosi_2011}, factorisation homology~\cite{ginot2014higher}, and
formality results~\cite{Hinich2003}, among others.

While Hochschild homology of an
associative 
$k$-algebra $A$ is in general difficult to 
%%Add some papers that have difficult computations
compute, in the case where $k$ is a field of 
characteristic zero and $A$ is commutative and 
smooth (for example if it is the coordinate
ring of a smooth algebraic variety) the 
celebrated Hochschild--Kostant--Rosenberg 
(HKR) theorem~\cite{HKR} identifies the Hochschild 
homology of $A$ with its module $\Omega_{A}^*$ 
of algebraic differential forms, which is 
nothing but a free commutative algebra over
the module $\Omega_{A}^1$ of K\"ahler differentials of $A$. 
In fact, this result is also used the other way 
around: it provides us with a way to generalise
geometrical results, usually stated in terms
of differential forms and fields on manifolds,
to non-commutative or non-smooth algebras by 
replacing these geometrical objects with 
Hochschild homology and cohomology. This
philosophy falls under the general theory
of non-commutative geometry; see~\cite{Dolgushev2010,Tsygan2005,GraciaBonda2001,CONNES2008}.

The HKR theorem depends on two hypothesis on the
underlying algebra that are of very different 
flavours: while smoothness
is a property that concerns certain geometric 
regularity of the algebra itself, and is thus
intrinsic to the category of commutative algebras,
the constraint that the associative algebra be 
commutative for one to obtain a description
of its cohomology involves, perhaps in a 
more mysterious way, the interplay between the 
category of commutative algebras and the one of 
associative algebras. 

Recently, V.~Dotsenko and the second author
have shown in~\cite{PBW} that one can produce,
using the language of operads, what 
they consider the `bare-bones' framework for 
Poincar\'e–Birkhoff–Witt (PBW) type theorems
about universal enveloping algebras of types
of algebras. There, they have shown that one can  
understand PBW-type results by way of studying a
homological property between morphisms of 
operads: the universal enveloping algebra functor
associated to a map of algebraic operads
satisfies a PBW-type property if and only if
it makes its codomain a \emph{free} right module.

We pursue this philosophy here, by 
considering the question of the existence of 
an HKR-like theorem for operadic algebras.
%%%Introduce HKR theorem here
 Since
the ingredients we will need are slightly more involved
than those in~\cite{PBW}, let us first recall these.

Given an operad $\PP$ and an algebra $A$ over $\PP$, 
we can consider its \emph{cotangent homology}~\cite{Milles}, which
we will write $\HH_*^\PP(A,M)$, and which corresponds to the 
Hochschild homology when $\PP$ is the operad governing
associative algebras, to Chevalley--Eilenberg homology 
when $\PP$ is the operad governing Lie algebras, and
to Harrison homology when $\PP$ is the operad governing 
commutative algebras over a field of characteristic zero.
Any map of algebraic operads 
$f:\PP \longrightarrow \QQ$ induces a restriction
functor $f^*\colon \QQ\hy\Alg \longrightarrow \PP\hy\Alg$ 
and, in turn, a map
\[ \HH_*(f) : \HH_*^\QQ(A,A)\longrightarrow 
	\HH_*^\PP(A,A),\]
and an HKR theorem can be seen as a way to promote this
map to an isomorphism, by applying an appropriate
functor to the codomain; the resulting object in the codomain deserves to be thought as ``differential forms'' on $A$. With this in mind,
our first step towards obtaining an HKR-type
formalism is the following result. It says that
promoting the map $\HH_*(f)$ to a possible candidate
for an HKR isomorphism can be done as soon as one
produces a quasi-free resolution of $\QQ$ in \emph{left}
$\PP$-modules. This resolution will have the form
$(\PP\circ\YY,d)$ for some graded symmetric sequence $\YY$ 
of generators, and these will play the central role of the ``functor 
of differential forms for $f$''. We will see it is convenient to phrase
our result in terms of the complexes $\Def_*^\PP(A)$ and 
$\Def_*^\QQ(A)$ computing the homology groups above.

\begin{theorem*}
Let $\FF=(\PP\circ\YY,d)$ be a quasi-free resolution
of $\QQ$ in left $\PP$-modules. Then there exists
a functorial complex $\Omega_{\FF,A}^*$ of `differential forms' 
on $A$ associated to $f$, depending on $\FF$ and 
$\Def_*^\QQ(A)$, and a morphism
of complexes
\[ \HKR_{\FF,A} : \Def_*^\QQ(A) \longrightarrow \Omega_{\FF,A}^*.\]
\end{theorem*}

There is a well developed theory of K\"ahler differentials 
$\Omega_{A}^1$ for operadic algebras~\cite{LV}, that is, 
algebraic $1$-forms, which we use to construct the module 
of differential forms $\Omega_{\FF,A}^*$. The construction 
$\Omega_{A}^1\longmapsto \Omega_{\FF,A}^*$ is not intrinsic to 
$\PP$-algebras but depends on the homotopical properties 
of the morphism $\PP\longrightarrow \QQ$ and, up to 
quasi-isomorphism, on a choice of resolution $\FF$ of $\QQ$, as we 
explain in Section~\ref{sec:morphism}. 
With this construction at hand, we consider the notion of 
smoothness in Section~\ref{sec:smooth} as a generalisation of one 
of the equivalent notions of smoothness in the commutative case, 
which we recall from the excellent monograph~\cite{Loday1992}.
%cite original def.
This allows us to make the following definition, central to our
paper:

\begin{definition*}
The map $f:\PP\longrightarrow \QQ$ has the 
\emph{Hochschild--Kostant--Rosenberg property}
if for every smooth $\QQ$-algebra~$A$, the map $\HKR_{\FF,A}$
is a quasi-isomorphism.
\end{definition*}
 
The main result of this paper is that the PBW property is,
in the sense made precise below, Koszul dual to the HKR property,
as we record in Corollary \ref{cor:PBW}. We cannot avoid to note
this follows the `mantra' pursued by B. Ward in~\cite{BenWard},
that it is desirable to consider Koszul duality not as an aspect 
of categories separately, but rather as a construction which 
intertwines functors between them. 

\begin{theorem*}
Let $f\colon \PP \longrightarrow \QQ$ be a map of Koszul operads. Then $f$ has the HKR property if the morphism of Koszul dual 
operads $\QQ^! \longrightarrow \PP^!$ enjoys the PBW property.
\end{theorem*} 

This gives us a short and conceptual proof of the classical
HKR theorem: the maps $\mathsf{Ass}\longrightarrow \mathsf{Com}$ 
and $\Lie \longrightarrow \mathsf{Ass}$ are Koszul dual, so that
the classical PBW theorem implies, in this way, the HKR theorem. 
With generous hindsight, this comes as no surprise: apart from using standard
techniques of localisation to reduce the proof of the HKR
theorem to smooth local commutative algebras, 
a straightforward way to prove that the HKR theorem holds 
is by use of the Koszul complex of the symmetric algebra
$S(V)$ and its Koszul dual coalgebra $S^c(V[-1])$. 

Having settled the above, we then observe there 
are several examples of maps of operads satisfying the HKR property and, in Section \ref{sec:new HKR theorems}, we explore some of them. 
Of particular interest is the map $\mathsf{Perm}
\longrightarrow \mathsf{Com}$ which factors the projection
of the associative operad onto the commutative operad by passing
through permutative algebras~\cite{Chapo}. 
The Koszul dual map to the projection $\mathsf{Perm}
\longrightarrow \mathsf{Com}$ is the inclusion 
$\mathsf{Lie} \longrightarrow \mathsf{PreLie}$, which
is known to enjoy the PBW property by~\cite{PBW}. In 
Corollary~\ref{cor:perm} we conclude that the following
HKR-type theorem holds, providing us with 
the computation of the cotangent homology of a smooth commutative algebra $A$ seen as a permutative algebra. 

\begin{theorem*}
The permutative cotangent homology of a smooth
commutative algebra $A$ is given by a module
$\mathsf{RT}_{\neq 1}(\Omega_A^*)$ which is spanned 
by rooted trees whose vertices are labeled 
by elements of the classical space of
forms $\Omega_A^*$
and no vertex has exactly one child.
\end{theorem*}

Finally, we offer a technique to compute
the tangent cohomology of a $\PP$ algebra coming from 
a smooth $\QQ$-algebra under the projection 
$f:\PP\longrightarrow \QQ$ in case $\PP$ is obtained
from a filtered distributive law~\cite{FilteredLaw}
between
$\QQ$ and $\mathcal R$ as originally defined by V. Dotsenko 
in~\cite{Filtered}. The shining example of this phenomenon
is the way in which the operad $\A$ is obtained from $\C$ and
$\mathsf{Lie}$; in this way, the reader may think of the
following \emph{filtered HKR theorem} 
as another `ultimate' generalization to
algebraic operads of the classical HKR theorem for the
map $\A\longrightarrow \C$. Indeed, in this
case, the functor $\mathcal R^\antishriek$
below is precisely $V\longmapsto S^c(V[-1])[1]$.

\begin{theorem*}
Suppose $\PP$ is obtained from Koszul operads $\QQ$ and $\mathcal R$ by a filtered
distributive law, so that $\PP$ is isomorphic to $\QQ\circ\mathcal R$
as a right $\mathcal R$-module. Then for every smooth $\QQ$-algebra $A$
the cotangent homology of $f^*A$ is given by the endofunctor
\[ A\longrightarrow 	\mathcal R^{\emph\antishriek}(\Omega_A^1)\]
of ``$\mathcal R^{\emph\antishriek}$-enriched differential forms'' on $A$.
\end{theorem*}

 Dually to the result of homology, we were
 able to obtain a result for tangent
 cohomology. In this case, a choice of
 quasi-free resolution $\FF$
 gives us a functor of ``poly-vector fields''
 ${A\longmapsto \Tpoly A}$, and we obtain the
 following:
 
\begin{theorem*} 
	If $f$ satisfies the HKR property	then for every smooth 
	$\QQ$-algebra $A$ there is a quasi-isomorphism of complexes:
	\[ 
		\HKR^A :  
			\Def^*(f^*A) \longrightarrow\Tpoly A 
				.
								\] 
\end{theorem*}

In case of cohomology, 
our result on filtered distributive
laws says that the tangent cohomology of $f^*A$ 
 is given by the endofunctor $\mathcal R^!(\Der(A)[1])[-1]$. In the classical
 case, we recover the Lie structure on tangent homology, since 
 $\Der(A)$ is a Lie algebra. Indeed, since $\mathcal R^!=\C$, 
 the usual distributive law allows us to give 
 $\mathcal R^!(\Der(A)[1])[-1]$
 a Lie algebra structure isomorphic to the one on $\HH^*_\A(A)$. It is unclear, however, how
 one could attempt to obtain the Lie algebra
 structure on tangent cohomology in a more
 general situation. 
 \medskip
 
\textbf{Structure.} The paper is organised as follows.
In Section \ref{sec:com HKR} we 
recall the usual HKR theorem for
smooth commutative algebras, in a 
way that suits our operadic approach 
that follows, and hoping that it 
will be useful for the reader to
incorporate the new formalism
that we then develop in
Section~\ref{sec:main}.
Here, we recall the notions of 
(co)tangent (co)homology and 
introduce the relevant notions
of smoothness and the ``full'' module of 
differential forms. With this at
hand, we introduce the HKR property 
and prove our main theorem.
In Section~\ref{sec:new HKR theorems}, 
where we focus on applications, we 
show how to recover the classical 
HKR theorem from our main result and 
apply it to obtain new examples:
we obtain a ``permutative'' HKR
theorem for smooth commutative algebras
and a ``pre-Lie'' HKR theorem for
smooth braces algebras. In the process of
drawing some connections of our work 
to that of J.~Griffin~\cite{Griffin},
we obtain an HKR theorem for operads
obtained from filtered distributive
laws, and briefly outline how it recovers
the HKR isomorphism at the level of
Lie algebras. 
Finally,  with the purpose of making
this paper better self-contained, 
we collect some useful results in
an Appendix about algebraic operads,
their algebras and their K\"ahler 
differentials, hoping it will be
of use for a reader with some 
background in algebraic operads.

\medskip

\textbf{Notation and conventions.} For references
on operads and their modules we point the reader to~\cite{Fresse,LV},
and to~\cite{Weibel,CarEil,Loday1992} for homological algebra. We allow 
operads to be homologically graded, but will make it clear when we 
require operads to be dg. We assume that algebras over operads 
are non-dg, and we fix a closed symmetric monoidal category 
$\mathsf{C}$ like $\mathsf{Vect}$ over which our algebras
are defined; we always work over a field of zero characteristic. 
Most arguments we make actually hold for dg algebras, taking 
into the account the given bigrading of the resulting objects.
For simplicity, we work with non-dg algebras. We write $\#$ 
for the forgetful functor from algebras to~$\mathsf{C}$.
If $V$ is a chain complex and $p\in\ZZ$, we write $V[p]$ for the 
chain complex for which $V[p]_n = V_{n-p}$ for each $n\in\ZZ$, and 
whose differential changes sign according to the parity of $p$. 
Accordingly, if $\QQ$ is an operad, we write $\QQ\{p\}$ for the
operad uniquely defined by the condition that a $\QQ\{p\}$-algebra 
structure on $V$ is the same as a $\QQ$-algebra structure on $V[p]$. 

Throughout, for two quadratic
operads given by quadratic data
$(V,R)$ and $(V',R')$, we will
\emph{only} consider maps of operads
induced by a map of quadratic
data $V\longrightarrow V'$
such that the induced map
$\mathcal T(V)\longrightarrow 
	\mathcal{T}(V')$
	sends $(R)$ to $(R')$.
We remind the reader that the
data $(V,R)$ may contain \emph{non-binary} generators and that, in
this case, the weight and arity
gradings in $\mathcal T(V)$ may not
coincide; see~\cite{LV}*{\S 7.1.3}. Moreover, we 
confine ourselves to the category of weight graded operads
and their weight graded algebras and modules. 
We distinguish the weight degree from the homological
degree by using parentheses. Hence, while $\XX_3$ denotes
a  component of homological degree $3$, we write
$\XX_{(3)}$ for a  component of weight degree $3$.

\medskip

\textbf{Acknowledgements.} 
We kindly thank B. Keller for explaining to us the very
short proof of Lemma~\ref{lema:BK} which we reproduced here. 
We also thank V. Dotsenko, J. Bellier-Mill\'es,
N. Combe and J. Nuiten  for useful
conversations, comments and suggestions.

\section{The case of commutative algebras}\label{sec:com HKR}

This section serves to recall the objects and results related to the classical Hochschild--Kostant--Rosenberg theorem for commutative algebras. Such objects will be presented in the way that we find best suited for the operadic generalisation and the main results appearing in Section \ref{sec:main}. For a classical approach we recommend both Chapter 3 and Appendix~E~of~\cite{Loday1992}.

\subsection{The classical HKR morphism}

Throughout, fix a non-unital commutative
algebra $A$, and let us recall how to
construct a natural map that relates
the homology of $A$ as a commutative
algebra, its Harrison
homology, and the homology of $A$ as
an associative algebra, its Hochschild 
homology, through a particular functor. 
This is the well known 
\new{Hochschild--Kostant--Rosenberg map}
\[ \HKR_A : C^*(A,A) \longrightarrow 
	\Omega_A^*\]
where the left hand side is the cyclic
Hochschild complex of $A$ considered
as an associative algebra and $\Omega_A^*$ is the space of \new{differential forms} on $A$. We now recall the details
necessary to construct this map.

For any commutative algebra $A$, the
module of \new{K\"ahler differentials} 
$\Omega_A^1$ of $A$ is the symmetric
$A$-bimodule representing the
functor of derivations 
\[ M\longmapsto \Der(A,M).\]
Recall that we have a natural 
isomorphism of symmetric $A$-bimodules
\[ I/I^2\longrightarrow \Omega_A^1,\quad \text{where}
	\quad I = \ker(\mu:A\otimes A\longrightarrow A)\]
such that $1\otimes x-x\otimes 1 +I^2\longmapsto dx$. 

\begin{definition}
Let be $J$ the kernel of the multiplication of a cofibrant 
replacement $\cof A$ of $A$. The \new{cotangent complex} 
of $A$ with coefficients in a symmetric $A$-bimodule $M$
is by definition 
 \[\Def_*(A,M) = J/J^2\otimes_{\cof A} M.\] 
The cotangent homology of $A$ with coefficients in
$M$ is, by definition, the homology of this complex,
and we write it $\HH_*(A,M)$. 
\end{definition}

In other words, this is the non-abelian 
derived functor of $M\longmapsto \Omega_A^1\otimes_A M$.
Dually, we have a \new{tangent complex}
of $A$ with values in $M$
\[\Def^*(A,M) = \hom_A(J/J^2,M)\]
and the tangent cohomology of $A$ with values in
$M$ is, by definition, the homology of this complex,
and we write it $\HH^*(A,M)$. 

\begin{definition} We say $A$ is a
\new{smooth commutative algebra} if
for every $A$-module $M$,
\[ \HH^1(A,M) = 0. \] 
\end{definition}

For our convenience and that of the
reader, we record now some equivalent
definitions of smoothness, which in
particular show
that the cotangent homology
$\HH_*(A,A)$ of $A$ is very simple
in case it is smooth:
it is concentrated in degree zero
where it equals the module
of K\"ahler differentials
$\Omega_A^1$.  

%\textit{Remark.} One can relax the 
%hypothesis on the field $\kk$
%to work instead with flat algebras 
%over a commutative Noetherian ring, 
%but we will not concern ourselves 
%with this here. 

We remind the 
reader from the Appendix that
one can also consider relative versions
of the homology and cohomology theories 
above for a morphism of algebras. 
In particular, since $A$ is a commutative algebra,
we can consider the (co)homology of $A$
relative to $A\otimes A$ through the
multiplication map.  
\begin{prop}\label{thm:variants of smoothness}
Let $A$ be a finitely generated commutative
algebra over a field of characteristic zero,
and let $B=A\otimes A$.	
Then the following conditions are equivalent:
\begin{tenumerate}
\item $\HH^1(A,M)=0$ for any symmetric
$A$-bimodule $M$,
\item $\HH_1(A,A) =0$ and $\Omega_A^1$ is
a projective $A$-module,
\item $\HH^2(A\vert B,N) = 0$ for any $A$-module $N$,
\item $\HH_2(A\vert B,A) = 0 $ and $\Omega_A^1$ is
a projective $A$-module. 
\end{tenumerate}
\end{prop}

\begin{proof}
See~\cite{Loday1992}*{Appendix E}.
\end{proof}

{{}} As we noted, 
if $A$ is a smooth commutative algebra,
the fact $\Omega_A^1$ is projective,
implies that $\HH_*(A,A)$ is 
concentrated in degree zero and
\[ \HH_0(A,A) = \Omega_A^1. \]
This receives a map $A\longrightarrow 
\Omega_A^1$, the \new{universal
derivation}, and we can then
form the non-unital
symmetric algebra $S_A(\Omega_A^1[-1])$
of $\Omega_A^1$ under $A$. 
We call this the space of
differential forms on $A$
and write it $\Omega_{A}^*$. 
 Finally, let us recall
that the \new{cyclic Hochschild complex} 
$C_*(A,A)$ of $A$ is given for each
$n\in\NN$ by
$C_n(A,A) = A\otimes \overline{A}^{\tt n}$
and we write a generic element in
here by $a[a_1\vert\cdots\vert a_n]$.
There is a map of complexes
\[ \HKR_A : C_*(A,A) \longrightarrow \Omega_A^* \]
such that $a[a_1\vert\cdots\vert a_n] \longmapsto ada_1\cdots da_n$. It will be useful to note
that this map is the identity of $A$ in degree $0$,
and in fact split as a map of complexes as a sum of
this map and the remaining part 
\[ C_*(A,A)^+ \longrightarrow \overline{S}_A(\Omega_A^1).\]
where the right hand side uses the \emph{non-unital}
symmetric algebra functor under $A$. Over a field of characteristic zero, the HKR
map is a split injection. The \new{Hochschild--Kostant--Rosenberg theorem}~\cite{HKR} asserts the following stronger conclusion in case $A$ is smooth:

\begin{theorem}\label{thm:classical HKR}
For every smooth commutative algebra $A$
of finite type over $\kk$ 
the morphism
\[ \HKR_A : C_*(A,A) 
	\longrightarrow \Omega_A^*\]
is a quasi-isomorphism.  \qed
\end{theorem}

\subsection{The HKR isomorphism}
{{}} There are many proofs of 
Theorem~\ref{thm:classical HKR} in the literature. Because it serves
to illustrate the general formalism that we will
develop later, let us
give a non-standard
proof of this theorem: it will follow once
we show that (in the dg setting), the HKR
map is a quasi-isomorphism for cofibrant
algebras that resolve smooth 
algebras. 

{{}} 
To see why this is enough,
observe that $\HKR_A$ is natural, 
in the sense 
that given a map of algebras $f\colon B 
\longrightarrow A$ we have that 
$\HH_*(f)\circ \HKR_B = 
		\HKR_A \circ\, \Omega^*_f$
		or, what is the same, there is a commutative diagram 
		$$
\begin{tikzcd}
 C_*(B,B)\arrow[d]\arrow[r] &
 	  C_*(A,A) \arrow[d] \\
  \Omega_B^*\arrow[r] & \Omega_A^*.
\end{tikzcd}
$$
Since the functor $\HH_*$, by its very definition,
preserves quasi-isomorphisms, and since
the $\HKR$ map is a quasi-isomorphism
for cofibrant algebras, we deduce the 
following interesting lemma. Its
content is central to develop
our operadic formalism later.

\begin{lemma}\label{cor:funcor}
The map $\HKR_A$ is a quasi-isomorphism
if and only if the functor of differential forms
$A\longmapsto \Omega_A^*$ preserves quasi-isomorphisms 
$Q\longrightarrow A$ for $Q$ a cofibrant resolution of 
an arbitrary smooth algebra~$A$. \qed
\end{lemma}
To proceed with the proof, let us first recall that we can
express $C_*(A,A)$ as a twisted tensor product $A\otimes \Bar A$
where $\Bar A  = (T^c(s\overline{A}),\delta)$ is the bar construction
of $A$, arising from the fact the associative operad is Koszul self-dual. Similarly, there is a commutative-Lie bar-cobar 
adjunction arising from Koszul duality between the category 
of conilpotent Lie coalgebras
and the category of commutative algebras 
\[
\mathcal C \colon \mathsf{Lie}\hy\mathsf{Cog} \rightleftarrows \mathsf{Com}\hy\mathsf{Alg} \colon \mathcal L.
\] 
The only properties of this adjunction that we need are the following:
\begin{titemize}
\item the counit of the adjunction $\mathcal C \mathcal L \to \id$ is a quasi-isomorphism and,
\item the commutative cobar construction of a Lie coalgebra is given by the quasi-free commutative algebra $\mathcal C (\mathfrak g) = (S(\mathfrak g[1]),\delta)$, where $\delta$ is a differential extending the cobracket of $\mathfrak g$. 
\end{titemize}
These two adjunctions interact in the following way: the 
restriction functor  from conilpotent associative coalgebras
to conilpotent Lie coalgebras $\mathsf{Ass}\hy\mathsf{Cog}
\longrightarrow \mathsf{Lie}\hy\mathsf{Cog}$
has a left adjoint~$\mathcal U^c$, the \emph{universal enveloping 
coalgebra}, which satisfies $\Bar \pi^* = \mathcal U^c \mathcal L $, 
where $\pi^*$ is the forgetful functor from commutative algebras to 
associative algebras, see Lemma~\ref{lem:fundamental lemma}.

\begin{lemma}
The $\HKR$ map  is a quasi-isomorphism
for any free commutative dga algebra.
\end{lemma}

\begin{proof}
A derivation on a free commutative algebra $A=S(V)$ is uniquely 
determined by specifying the values on generators, from where 
it follows that $\Omega_A^1 = A\otimes V$. 
In this way, we obtain an identification
\[ \Omega_A^* = A\otimes S^c(V[-1]).\]
At the same time, $A=\mathcal{C}(V[-1])$ is the 
commutative cobar construction of the abelian  Lie coalgebra $\g = V[-1]$.
It follows that as chain complexes
\begin{align*}
	 C_*(A,A) &= A\otimes \Bar\mathcal C(V[-1])  \\
	&=  A\otimes {\mathscr U^c \mathcal L \mathcal C(V[-1])} \\
	&\simeq  A\otimes \mathscr U^c(V[-1])  \\
	&= A\otimes S^c(V[-1]).
	\end{align*}
On the third line we used that $\mathscr U^c$ preserves quasi-isomorphisms, 
which is a consequence of the PBW theorem \cite{michaelis1980lie}. The resulting 
anti-symmetrization quasi-isomorphism $\varphi_A :\Omega_A^* 
\longrightarrow C_*(A,A)$
gives us an inverse to
$\HKR_A$ when taking homology, which
proves our claim.
\end{proof}

A \emph{cobar algebra} is any  commutative algebra $\mathscr{C}(\g)$ obtained via a cobar construction of a (shifted) Lie coalgebra $\g$. Recall that cobar algebras are triangulated and 
hence cofibrant and that every algebra admits a cofibrant replacement given by a cobar algebra,
as explained in Corollary 11.3.5 and Proposition B.6.6 of \cite{LV}.

\begin{prop}\label{prop:cofibrant HKR => qi}
	The map $\HKR_A$ is a quasi-isomorphism for any commutative cobar algebra.
\end{prop}

\begin{proof}
 The differential of the cobar construction $A=\mathscr{C}(\g)$ splits as $d_\g + \delta$.
As in the previous lemma, we have that 
\[\Omega_A^1 = (A\otimes \g[-1],d_A\otimes 1 + 1\otimes d_\g + \overline\delta),\] where we notice that the differential has an external component induced by $\delta$, such that if we
write $\delta x= x_{(1)}\otimes x_{(2)}$ in Sweedler notation,  then $\overline \delta (a\otimes x) = a x_{(1)} \otimes x_{(2)}$.
On the other hand, the same computation as in the previous lemma shows that there is a morphism of 
 $A$-modules, 
\[ \varphi_A : \Omega_A^* \longrightarrow
C_*(A,A)  .\]
The result now follows from taking the spectral sequence 
associated to the PBW filtration: the associated
morphism to $\varphi_A$ is equal, in
homology, to the desired inverse of the HKR map corresponding to the
commutative algebra $S(\g^\circ)$ where $\g^\circ$ is the Lie algebra $\g$
with zero bracket.
\end{proof}

\begin{theorem}
The $\HKR$ map is a quasi-isomorphism
for any smooth commutative algebra.
\end{theorem}

\begin{proof}
Let $A$ be a smooth commutative
algebra. Since $A$ is smooth it follows in
particular that $\HH_1(A,M)=0$ for
every $A$-module $M$. 
Setting $A=M$
we see that $\HH_*(A,A) = \Omega_A^1$,
from where it follows that for any
cofibrant replacement $p:Q\longrightarrow A$
\[ \Omega_p^1 : A\otimes_Q \Omega_Q^1
	\longrightarrow \Omega_A^1\] 
	is a quasi-isomorphism: the left
	hand side computes $\HH_*(A,A)$ and
	the induced map is then an isomorphism.
	Because the canonical map
	\[ \Omega_Q^1 = Q\otimes_Q \Omega_Q^1\longrightarrow  A\otimes_Q \Omega_Q^1
\] 
is a quasi-isomorphism, it follows that
$\Omega_Q^1\longrightarrow\Omega_A^1$ 
is a quasi-isomorphism. In view of
Lemma~\ref{cor:funcor}, we see that
the map $\HKR_A$ is a quasi-isomorphism.
\end{proof}

{{}} Before moving on, we would like to highlight the following
three points that will be revisited when we
develop the general operadic formalism for HKR
theorems:
\begin{tenumerate}
\item For $q: Q\longrightarrow A$ a cofibrant resolution of $A$ a smooth algebra
the map $q_!\Omega_Q^1\longrightarrow \Omega_A^1$
is a quasi-isomorphism. This is \emph{intrinsic} 
to the category
of commutative algebras and thus independent of the
map of operads $\A\longrightarrow \C$.
\item  Showing that for any cofibrant
algebra $Q$ the map 
$\Def_*(f^* Q,f^*Q) \longrightarrow \Omega_Q^*= S_*^c(\Def_*(Q))$
is a quasi-isomorphism is independent of 
smoothness, and depends on the map of operads $f:\A\longrightarrow \C$. In particular, we can
take ``affine'' algebras as test algebras
in this step, which we called ``cobar algebras''
above.

\item The way the two previous points are put
together is by noting that, since the functor
$S^c$ preserves
quasi-isomorphisms, we have that
$ \Omega_q^*:\Omega_Q^* \longrightarrow
  \Omega_A^*$
 is a quasi-isomorphism. Here it is the only
 step where we use the universal enveloping algebra
 functor associated to $\L\longrightarrow \A$ preserves
 quasi-isomorphisms, by the classical PBW theorem. 
\end{tenumerate}

\section{The HKR theorem for general operadic algebras}\label{sec:main}

In this section we generalise the classical notions from the previous 
section to algebras over operads, and prove Theorem~\ref{thm:morphism}, 
generalising the classical Theorem \ref{thm:classical HKR} to smooth 
algebras for morphisms of operads satisfying a natural homological 
condition.

\noindent \textit{Some conventions for this section.}

\begin{titemize}
\item We fix once and for all a morphism of non-dg Koszul
operads $\PP\longrightarrow \QQ$ which we assume comes
from a map of quadratic data.
All (co)operads will be homologically 
graded with zero differential and by $A$
we will always denote a $\QQ$-algebra.

\item For any operad $\OO$, we denote 
by  $\cof\colon \OO\hy\Alg\longrightarrow
\OO\hy\Alg$ a fixed choice of 
cofibrant replacement functor for $\OO$-algebras. 

\item When context allows, we will usually simply
write $\cof$ for a cofibrant replacement of some
algebra $A'$ which will be clear from context. 

\item In
particular, we will sometimes need to use the  composition $UQ$ where $U=U_\OO\colon 
\OO\hy\Alg\longrightarrow \A\hy\Alg$ is the 
associative universal envelope functor, in which
case we will usually to write this $UQ$ when the
algebra we applied it to is clear from context.
\end{titemize}

\subsection{Deformation complexes and (co)homology}

In this section we introduce the 
formalism of (co)tangent homology
and cohomology for algebras over an operad. 
We refer the reader to the
article~\cite{Milles} of J. Mill\`es for a thorough and comprehensive
study of this theory, and point
to the Appendix, where some useful
recollections on algebra over operads
and their operadic modules is given.
The reader can consult 
 Appendix~\ref{app:derivations}
for details on derivations and 
associative universal envelopes of
algebras over operads.

\begin{definition}
Let $A$ be a $\QQ$-algebra.  We define the  \new{tangent 
complex of $A$} with values in an 
operadic $A$-module $M$ by
\[ \Def^*(A,M) = \Der(\cof A,M)=\hom_{U\!\cof A}(\Omega_{\cof A}^1,N). \]
\end{definition}

Note that we dropped the subscript $\QQ$, which
will be clear from context.  
Observe that this is defined up to natural
quasi-isomorphism, and is well-defined in
the derived category of complexes.
 
The morphism $f:\PP\longrightarrow\QQ$ induces a restriction functor
$ f^* : \QQ\hy\Alg
	\longrightarrow 
		\PP\hy\Alg$
that assigns $A$ to the $\PP$-algebra 
$f^*A$ with the same underlying object
as $A$ along with 
the $\PP$-algebra structure
given by the composition
$ \PP \longrightarrow 
	\QQ \longrightarrow \operatorname{End}_A$. 
Observe that we can take 
$Q(f^*A)$ as a 
cofibrant replacement
$Q(f^*A) \to f^*(QA)$
of the  $\PP$-algebra $f^*(QA)$. 
In this way, we obtain a
natural map
\[\Def^*(f,M): 
	 \Def^*(A,M)
	 \longrightarrow
	 \Def^*(f^*A,f^*M) 
	 \] 
for every operadic $A$-module
$M$.  

{{}} 
\begin{definition}\label{def:tangent}
The cohomology of $\Def^*(A,M)$
is, by definition, \new{the tangent cohomology of 
$A$ with values in $M$}, and we will write it
$
	\HH^*_\QQ(A,M)$. 
	\end{definition}

\begin{remark}
When $A$ is an associative algebra, 
$\HH^*(A,M)$
differs from the classical Hochschild cohomology 
groups $\HHC^*(A,N)$ of $A$ only in
that we do not quotient out by inner derivations
in degree zero and we discard the $0$th classical
Hochschild cohomology group of $A$ with values
in $N$ so that
\[
\HH^*(A,M) = \begin{cases}
	\HHC^{*+1}(A,M) & \text{if $*\geqslant 1$}, \\
	\Der(A,M) & \text{if $*=0$}.
		\end{cases}
			\]
\end{remark}

In a similar fashion to tangent cohomology, we define the cotangent
complex.

\begin{definition}\label{def:cotangent}
The \new{cotangent 
homology of $A$ with coefficients in $M$} through the 
\new{cotangent complex of $A$} which is obtained as 
\[ \Def_*(A,M) = \Omega^1_{\cof A}\otimes_{U\!\cof A} M \]
and write it $\HH_*^\QQ(A,M)$. 
\end{definition}

Note that from 
Proposition~\ref{par:CCiso}, in case we take 
$f:\A\longrightarrow \C$, $A$ a commutative 
algebra and choose  $\Omega BA$ for the cofibrant
resolution for the associative algebra $A$, we 
have that
\[ \Def_*(f^*A,f^*A) = s^{-1}C_*(A,A)^+. \]

{{}} 
It is useful to note there are
universal coefficients for these homology theories~\cite{Fresse}. Indeed, writing the functors
as compositions,  where
we write $U$ for the
associative enveloping algebra $UA$ to lighten the notation
\[\Def^*(A,M) 
	= \hom_U(\Def_*(A,U),M),
	\]
	\[\Def_*(A,M) 
	= \Def_*(A,U)\otimes_U M
\]
we obtain two 
\new{universal coefficient spectral
sequences},
\[ E_{s,t}^2 = 
\Tor_s^{U}(-,\HH_t^\QQ(A,U)) \Longrightarrow \HH_{s+t}^\QQ(X,-), \]
\[ E_2^{s,t} = \Ext^s_{U}(\HH_t^\QQ(X,U),-)\Longrightarrow	
	\HH^{s+t}_\QQ(A,-), \]
that explain the relation between (co)homology
theories given by $\Ext$ and $\Tor$ functors and
the operadic theories. For example, in case 
we do this for the associative operad, we 
observe that \[\HH_t^\A(A,UA)=0\quad \text{for
$t\geqslant 1$,}\]
which shows Hochschild 
(co)homology of associative algebras
is given by $\Tor$ and $\Ext$ functors. 

When $A$ is commutative, the cotangent
homology groups, usually known
as the Andr\'e--Quillen homology groups
$\HH_*^\C(A,UA)=\HH_*^\C(A,A)$
are in general non-zero
in higher degrees, so there are obstructions to this comparison.
%probably only in positive char?

\subsection{Smooth algebras}\label{sec:smooth}
In analogy with the characterization of smoothness 
for commutative algebras of Theorem \ref{thm:variants of smoothness}, we introduce the following definition.
\begin{definition} The $\QQ$-algebra
 $A$ is \new{smooth} if for every 
 operadic $A$-module~$M$, 
\[ \HH^1_\QQ(A,M) = 0. \]
or, what is the same, 
if $\HH^0_\QQ(A,-)$ is an exact functor in operadic
$A$-modules.   
\end{definition}

Although our focus lies on non-dg
algebras, it is useful to remark that
the definition is in
general \emph{not} invariant
under quasi-isomorphisms. Indeed, suppose that
$ q: A\longrightarrow A'$ is a quasi-isomorphism
 of dg $\QQ$-algebras and let us
assume first that $A'$ is smooth, and let
$M$ be an operadic $A$-module. There is a
map
\[ \HH^1_\QQ(A,M) \longrightarrow \HH^1_\QQ(A',\psi_!(M)) = 0 \]
but, unless we assume that $q_!$ is well-behaved (that is, flatness assumptions
on $q$), there is no reason to expect this
to be an isomorphism.
However, if $A$ is smooth then any cofibrant
replacement of $A$ is one of $A'$, so that
in this case $A$ smooth implies $A'$ smooth.
It is immediate that 
every free (i.e. affine) $\QQ$-algebra is smooth.

Let us now consider a related condition: 
we say that $A$ is \new{quasi-smooth}
if for 	some ---and hence, every--- 
cofibrant replacement $p: \cof A\longrightarrow A$, the induced map on 
Kahler differential forms
	\[\Omega_p^1: p_!\Omega_{\cof A}^1 
		\longrightarrow \Omega_A^1\]
is a quasi-isomorphism of operadic 
$A$-modules. Before relating
the notions of smoothness
and quasi-smoothness, we
record the following lemma:
 	
\begin{lemma}\label{lema:BK} 
Let $q : X\longrightarrow Y$ be a map 
of complexes of operadic $A$-modules, and suppose that 
for every operadic $A$-module $M$ 
the map
\[ q^* : \hom_A(Y,M) \longrightarrow \hom_A(X,M) \]
is a quasi-isomorphism. Then $q$ is a quasi-isomorphism.
\end{lemma}

\begin{proof}
Let us take $J$ an injective cogenerator of the category of left
$UA$-modules. Then the $p$th cohomology group of $\hom_A(X,J)$ 
identifies with \[F_p(X) :=\hom_A(H_p(X),J)\] because $J$ is 
injective. Since $J$ is also a cogenerator, the collection of functors 
$\{F_p\}_{p\in\ZZ}$ detects quasi-isomorphisms, which gives what we wanted.
\end{proof}

\textit{Remark.}
Observe that the reverse
implications is not true. Indeed,
let us consider the commutative 
algebra $A=k[x]$, the trivial
$A$-module $M=Y= \kk$ and the complex
$ X : A\longrightarrow A$
where the differential is given
by multiplication by $x$.
Then the quotient map 
$q: X\longrightarrow Y$
is a quasi-isomorphism, but the induced
map
$\hom_A(Y,M) 
	\longrightarrow \hom_A(X,M)$
	is not: the right hand side 
	computes $\Ext_A^*(\kk,M)$,
	and this may not always 
	be concentrated
	in degree $0$.
	
{{}} With this lemma at hand,
we can prove the following proposition. It is
interesting to compare it with Corollary 7.3.5 
in~\cite{hinich1997homological}. While we make
a statement about the behaviour of the
induced morphism 
\[ \Omega_p^1 : p_!\Omega^1_B
	\longrightarrow \Omega_A^1\] when
$f$ is an acyclic fibration onto a smooth algebra, 
that corollary makes
a statement about the behaviour of
that map when $f$ is an acyclic
cofibration; in both cases the conclusion
is that the map induced is
a quasi-isomorphism.
 
\begin{prop}\label{prop:reg}
Every smooth $\QQ$-algebra is
quasi-smooth.
\end{prop}
\begin{proof}
Suppose that $A$ is smooth, and let
$Q\longrightarrow A$ be a cofibrant 
replacement, let us show that the map
$p_!\Omega_Q^1\longrightarrow \Omega_A^1$
is a quasi-isomorphism of operadic $A$-modules. 
By the previous two lemmas, it suffices to
show that for every operadic $A$-module $M$, 
the induced map
\[  p^*: \hom_A(\Omega_A^1,M)
	\longrightarrow \hom_A(p_!\Omega_Q^1,M)
	  \]
is a quasi-isomorphism. By adjunction,
the codomain is naturally isomorphic to
\[ \hom_{UQ}(\Omega_Q^1,p^*M) \]
so we obtain $p^*$ identifies naturally with the map
$ \hom_A(\Omega_A^1,M) \longrightarrow 
\hom_Q(\Omega_Q^1,M)$
representing the pullback along $p$
\[p^*: \Der(A,M) \longrightarrow \Der(Q,p^*M).\]
Since $A$ is smooth
and the right hand side
computes $\HH^*_\QQ(A,M)$, it follows that this 
map is a quasi-isomorphism: it induces the
identity of $\HH^0_\QQ(A,M)=\Der(A,M)$. 
We conclude that 
$\Omega_p^1$ is a quasi-isomorphism, which
means that $A$ is quasi-smooth,
as we wanted. 
 \end{proof}
 
\begin{remark} It is
important to observe that the
notion of quasi-smoothness 
may be quite weak. 
For example, every
\emph{unital} associative algebra is
quasi-smooth, owing to the
fact that the module of
associative K\"ahler preserves quasi-isomorphisms. Indeed,
in this case this functor fits into an exact 
sequence
\[ 0 
	\longrightarrow	
	\Omega_A^1 
		\longrightarrow
		A\otimes A
		\longrightarrow A
			\longrightarrow 
			0 \]
and the second and third
functor clearly preserve
quasi-isomorphisms. 
 \end{remark}
 
 \subsection{Left Koszul morphisms}

In this section we collect some facts about 
(left) Koszul morphisms between weight graded operads,
which we introduce. We refer the reader
to the excellent monograph~\cite{PoliPosiAMS}*{Section 2.5}
for the case of algebras. We say a symmetric sequence $\XX$
 is \new{diagonally pure} if for each $p\in\NN$ the component~ $\XX_p$ of homological degree $p$ is concentrated in weight $p$. 
With this at hand, let us introduce the kind of left dg $\PP$-modules
that interest us.

\begin{definition}
A quasi-free left dg $\PP$-module $(\PP\circ \YY,d)$ 
is \new{diagonally pure} if its generating sequence $\YY$~is so. A
 map of operads $f: \PP\longrightarrow\QQ$ is \new{left
Koszul} if $\QQ$ admits a diagonally pure quasi-free resolution in
the category of  left $\PP$-modules. Finally, a resolution
$\mathcal F$ is \new{minimal} if the differential
of $\kk\circ_\PP \mathcal{F}$ vanishes.
\end{definition}

It is easy to see that any diagonally pure quasi-free 
resolution $(\PP\circ \YY,d)$ 
is minimal. Indeed, the resulting complex is of the form
$(\YY,\bar d)$. Since $d$ preserves the weight degree but
lowers the homological degree, the differential $\bar d$ does
too and, since $\YY$ is diagonally pure, $\bar d$ vanishes.

 \begin{definition}
We define $\Tor^\PP(\kk,\QQ)$ as the homology of 
 the dg module $\kk\circ_\PP \mathcal{F}$ where $\mathcal F$
 is any quasi-free resolution of $\QQ$ in left dg $\PP$-modules. 
 For each $i,j\in\PP$ we write $\Tor^\PP_i(\kk,\QQ)_{(j)}$ for the
 component of $\Tor$ in homological degree $i$ and weight degree $j$.
 \end{definition}
 
 It is useful to note that this is well defined, since the category
 of left $\PP$-modules admits a model structure in which the
 fibrations are the arity-wise surjections, the weak equivalences
 are the quasi-isomorphisms, and the quasi-free left modules
 are included in the class of cofibrant objects. It is 
 important that we are working over a field of characteristic
 zero, so that $\PP$ is $\Sigma$-cofibrant. %cite Fresse
 
 \begin{lemma}
 Let $f:\mathcal M\longrightarrow \mathcal N$ be a map of
 left $\PP$-modules and suppose that $\mathcal F = (\PP\circ \YY,d)$
 is a quasi-free complex mapping onto $\mathcal M$ and
 that $\mathcal R$ is a resolution of $\mathcal N$. Then
 there exists a map of left $\PP$-modules 
 $\mathcal F \longrightarrow  \mathcal R$ extending $f$, 
 and any two such choices are homotopic.\qed
 \end{lemma}

\begin{lemma}\label{lemma:diagonal} The map $f$ is left Koszul if and
only if $\Tor^\PP(\kk,\QQ)$ is concentrated on the diagonal. 
\end{lemma}

\begin{proof}
Let us show we can construct minimal
quasi-free resolutions $\FF = (\PP\circ\XX,d) \longrightarrow \QQ$.
To do this, let us consider an equivariant section $\sigma$ of the
projection $\QQ \longrightarrow \kk\circ_\PP \QQ$, which exists since we work over a field of characteristic zero, and let
$\XX_0 = \sigma(\kk\circ_\PP \QQ)$, so we have an epimorphism
\[ f_0:\PP\circ\XX_0\longrightarrow \QQ=\KK_{-1}.\] The kernel $\mathcal K_0$
of this map
is a left $\PP$-module, so we may repeat this and take $\XX_1$ a
minimal generating set for $\KK_0$ obtained from an equivariant
section of the projection 
$\KK_0\longrightarrow \kk\circ_\PP \KK_0$, 
along with $f_1 : \PP\circ \XX_1\longrightarrow \KK_0$.
Extend the 
construction above to $\FF_1 = \PP\circ (\XX_0\oplus s\XX_1)$ where
the differential is the unique map
$\XX_0\oplus s\XX_1\to \FF_1$
that vanishes on $\XX_0$ and
maps $s\XX_1$ onto $\KK_0$.
In this way, $\HH_0(\FF_1)$ is
isomorphic to $\QQ$ through
the map $f_0$. We can now continue
this process by adjoining generators
in homological degree $2$ to obtain
$\FF_2$ with $\HH_1(\FF_2)=0$
and $\HH_0(\FF_2)$ isomorphic to~$\QQ$. 
Continuing, in the limit, we obtain the desired
resolution.

Since the resolution is minimal, we see that $\XX$ is isomorphic
to $\Tor^\PP(\kk,\QQ)$, so $\XX$ must be concentrated in
the diagonal. Conversely, it is clear that if we have a diagonally pure
resolution, then $\Tor^\PP(\kk,\QQ)$ is concentrated
on the diagonal.
\end{proof}

Let us recall the following from~\cite{PBW} .

\begin{definition}
We say a morphism of operads $f\colon \PP \to \QQ $ enjoys the \new{PBW 
property} if there is an endofunctor 
$\mathcal X : \mathsf{C} \longrightarrow \mathsf{C}$
on the category underlying that of $\PP$-algebras so that for each $\PP$-algebra~$A$ there is a natural isomorphism
\[ f_!(A)^\# \longrightarrow \mathcal X(A^\#).\]
\end{definition}

The main result of~\cite{PBW}, if we
take monads there to be algebraic
operads, is the following:

\begin{theorem}
The morphism of operads $f\colon \PP \to \QQ$ satisfies the PBW 
property if and only if it makes $\QQ$ 
into a free right module over
$\PP$. In this case, the functor $\mathcal X$ is a basis for $\QQ$ as a right $\PP$-module.\qed
\end{theorem}

Our first main theorem shows the PBW property above
is related, by Koszul duality, to the notion of
left Koszul morphisms, at least when 
$\PP$ and $\QQ$ are Koszul operads. Note
this is an extension of~\cite{PoliPosiAMS}*{Corollary 5.9} 
to algebraic operads.

\begin{theorem}[Duality]\label{thm:duality}
	A map between Koszul operads is left Koszul if and only if
	its Koszul dual map satisfies the
	PBW property.
	\end{theorem}

To do this, we just need two technical
lemmas, beginning with the following
simple homological criterion
for freeness, which we recall from
Proposition 4.1 in~\cite{PBW}, for
example. We phrase it in a slightly different way than we did there:

\begin{lemma} A right $\PP$-module is 
$\mathcal{M}$ is free if and only if
for every $i,j\in\NN$, the group
$\Tor^\PP_{j-i}(\mathcal{M},\kk)_{(i)}$
vanishes unless $i=j$. \qed
\end{lemma}

Note we are simply saying that
the homology of $\mathcal{M}\circ_\PP^{\mathbb L} \kk$ is concentrated in
degree $0$. The second lemma relates the derived
functors $\kk \circ_\PP^\mathbb{L}\QQ$
and $\PP^! \circ_{\QQ^!}^\mathbb{L} \kk$
in case $\PP$ and $\QQ$ are Koszul 
operads. We point the reader to
Theorem 5.8 in~\cite{PoliPosiAMS} which
proves the result for associative
algebras.

\begin{lemma}
Let $f:\PP\longrightarrow \QQ$ be
a morphism of Koszul operads, and let
$f^\antishriek : \QQ^\antishriek \longrightarrow \PP^\antishriek$ be
its dual morphism. For each $j,i\in\NN$
we have a natural isomorphism:
\[ \Tor^{\QQ^\antishriek}_{j-i}(\PP^\antishriek,\kk)_{(i)} \longrightarrow 	\Tor^\PP_i(\kk,\QQ)_{(j)}^*.\] 
\end{lemma}

\begin{proof}
It suffices to note that, in the
category of right $\PP$-modules,
$\kk$ admits a resolution 
$\PP^\antishriek\circ \PP \longrightarrow \kk$
while, in the category of 
left $\QQ^!$-modules, $\kk$
admits a resolution 
$\QQ^! \circ \QQ \longrightarrow \kk$. 
Applying the functor
 $-\circ_\PP \QQ$ in the first case
and the functor $\PP^! \circ_{\QQ^!} -$ in the second case, we obtain two complexes
$\PP^\antishriek \circ \QQ$ and $\PP^!\circ \QQ$
 	that are related by the duality
 	described in the statement of the lemma.
\end{proof}

\begin{proof}[Proof of Theorem~\ref{thm:duality}]
The  previous three lemmas immediately imply
the result.
\end{proof} 

\subsection{The HKR morphism}\label{sec:morphism}

In this section we construct, for each map $f:\PP\longrightarrow
\QQ$ and each $\QQ$-algebra $A$, an operadic analogue of the 
classical HKR map . This map relates 
the deformation complex of the $\PP$-algebra $f^*A$ to a certain
space of `differential forms' on $A$ depending functorially,
as in the classical setting, on $\Omega_A^1$. 

\begin{remark} We cannot avoid making the point that, 
while the morphisms of operads in~\cite{PBW, dPBW}
enjoying the PBW property involve a statement about
the pushforward functor $f_!$ on $\PP$-algebras, 
the morphisms of operads we are interested in involve
a statement about the pullback functor $f^*$ on $\QQ$-algebras.
As mentioned in the introduction, this follows the `mantra'
promoted in~\cite{BenWard}.
\end{remark}

To begin, let us take a quasi-free resolution
$\FF=(\PP\circ\YY,d)$ of $\QQ$ in left $\PP$-modules.
Let us recall from Lemma~\ref{lema:isos}
that if $Q= (\QQ(V),d)$ is a cofibrant resolution of 
$A$, then $\Omega_Q^1\otimes_U Q$ is canonically 
isomorphic to $V\otimes \QQ(V)$, while $\Omega_{Q,A}^1$ 
is canonically isomorphic to $V\otimes A$. We will be interested
in applying the functor $\kk\circ_\PP \FF = \YY$ to the space
$\Omega_Q^1\otimes_U Q$, relative to the algebra $Q$, as 
the following definition explains. 

\begin{definition}
We define $\Omega_{\FF,A}^*$, the \new{space of differential forms
on $A$} associated to $\FF$, as the chain complex 
$\YY(V)\otimes \QQ(V)$. 
Its differential is the one induced from $Q$ and $\FF$.
\end{definition} 
%try to find categorical def

The following proposition shows 
that to each resolution we may associate an `HKR map'.

\begin{prop}\label{prop:natHKR}
There exists a map of chain complexes
\[ \HKR_{A} :  
	\Def_*(f^*A) 
 		\longrightarrow
 			 \Omega_{\FF,A}^*,\]
for every choice of resolution $\FF$. Moreover, if $Q\longrightarrow A$
is a cofibrant resolution of $A$ in the category of $\QQ$-algebras, there
is a commutative diagram
$$
\begin{tikzcd}
 \Def_*(f^*Q)\arrow[d]\arrow[r] &
 	  \Def_*(f^*A) \arrow[d] \\
  \Omega_{\FF,Q}^*\arrow[r] & \Omega_{\FF,A}^*.
\end{tikzcd}
$$
\end{prop}

Observe that, unlike the case where $A$ is a
commutative algebra, the source and the target
of the HKR map do not admit natural operadic
$A$-module structures. 
We call $\HKR_A$ the \new{Hochschild--Kostant--Rosenberg
map} associated to $f$ and the algebra
$A$. 
Observe, moreover, that $\HKR_A$ manifestly depends
on the resolution $\FF=(\PP\circ\YY,d)$ and
on the cofibrant resolution $Q$ of $A$.
This is not a problem for us, since 
the map it induces on homology does not. 
With this at hand, we can define the HKR property:

\begin{definition}\label{def:HKR property}
We say that $f$ satisfies the \new{HKR property}
if the map $\HKR$ is a quasi-isomorphism for every
smooth $\QQ$-algebra. 
\end{definition}

\begin{proof}[Proof Proposition~\ref{prop:natHKR}]
Let $A$ be any $\QQ$-algebra and let us pick a cofibrant 
resolution of $A$ of the form $Q=(\QQ(V),d)$. By the HKR 
property of $f$, the non-dg $\PP$-algebra $(\QQ(V),0)$ 
admits a cofibrant resolution of the form
$(\PP\circ\YY(V),d_V) 
		\longrightarrow \QQ(V)$
and hence perturbing this we obtain a resolution
\[ 
	Z = (\PP\circ\YY(V),d_V+\delta) 
		\longrightarrow (\QQ(V),d)
									\]
of the $\PP$-algebra $(\QQ(V),d)$. If we use this resolution
to compute the cotangent complex of the $\PP$-algebra $Q 
= (\QQ(V),d)$, we obtain a complex of the form
\[ 
	\Def_*(f^*A) = 
		\Omega_Z^1\otimes_U Q = 
			(\YY(V) \otimes \QQ(V) , \delta_1)
												\]
and, tautologically, the cotangent complex of the 
$\QQ$-algebra $A$ may be computed through the complex
\[ 
	\Def_*(A) = 
		\Omega_Q^1\otimes_U Q 
			= (V\otimes \QQ(V),\delta_1'). 
											\] 
The fact all of the constructions and isomorphisms above are 
natural, with the exception of the
perturbation process, means that, in fact, the complex $(\YY(V) \otimes 
\QQ(V) , \delta_1)$ 	is obtained from the complex	$(V\otimes 
\QQ(V),\delta_1')$ by the endofunctor $\YY$ relative to the 
algebra $\QQ(A)$. The commutativity
of the diagram is then immediate,
although we cannot promote
$\Omega_{\FF,A}^*$ to a bona-fide
functor and hence $\HKR_A$ to a
natural transformation.
\end{proof}

With this at hand, our second main result
for maps satisfying the HKR property is
the following `operadic HKR theorem', which
we now prove with a series of lemmas. Its 
immediate application is the corollary that
follows it, which we will use heavily later on.

\begin{theorem}\label{thm:morphism} 
Every left Koszul map between Koszul operads
has the HKR property.
\end{theorem}

\begin{cor}[PBW criterion]\label{cor:PBW}
Every map between Koszul operads
whose dual has
the PBW property satisfies the HKR property.
\end{cor}

\begin{proof}
This follows immediately from Theorem~\ref{thm:duality}.
\end{proof}

\begin{remark}
It is natural to wonder whether a converse to
this last corollary exists. The condition that
the resolution $\FF$ be diagonally pure
 makes a certain
spectral sequence collapse and gives our result,
modulo the computations and various lemmas
that we have made use of. In a generic case,
one should expect an ``HKR spectral sequence''
to exist coming from the resolution $\FF$
and, in favourable cases, one may obtain
an HKR theorem without requiring that
$\FF$ be diagonally pure. It would certainly
be interesting to have an example of
this behaviour.
\end{remark}

The proof of Theorem~\ref{thm:morphism} relies on the following 
fundamental lemma relating the two different constructions 
arising from the twisting morphisms $\phi$ and $\psi$ associated
to the Koszul operads $\PP$ and~$\QQ$, respectively.

\begin{lemma}\label{lem:fundamental lemma}
Let $\PP^{\antishriek}$ and $\QQ^{\antishriek}$ be the Koszul dual 
cooperads to $\PP$ and $\QQ$. For the commutative diagram of
maps of (co)operads and twisting morphisms
$$		
	\begin{tikzcd}
		\PP \arrow[r, "f"] & \QQ	\\
		\PP^{\emph\antishriek}
			\arrow[u,"\phi"] 
			\arrow[r,"g",swap] & 
		\QQ^{\emph\antishriek}
			\arrow[u,"\psi",swap],
	\end{tikzcd}
$$
there is a natural isomorphism of functors 
$\Bar_\phi f^*= g^!\Bar_\psi \colon \QQ\hy\Alg 
\longrightarrow \PP^\antishriek\hy\Cog$, where $f^*$ denotes the restriction of scalars functor and $g^!$ denotes the coinduction functor. 
\end{lemma}

\begin{proof}
Ignoring the additional differentials produced by the bar 
construction, $\Bar_\psi$ produces the cofree conilpotent 
$\QQ^\antishriek$-coalgebra functor on the underlying chain 
complex and $g^!$ is the right adjoint of the corestriction of 
scalars functor $g_*$. Since the composition of right adjoints 
is a right ajoint, we conclude that up to bar-differentials both 
$\Bar_\phi f^*$ and $g^!\Bar_\psi$ correspond to the cofree 
conilpotent $\QQ^\antishriek$-coalgebra on the underlying space.
The commutativity of the diagram above guarantees that both 
differentials are the same.
\end{proof}

\begin{cor}\label{cor:fund lemma}
In the conditions of the previous lemma, if the
coinduction functor $g^!$ satisfies the PBW property, 
there is a quasi-isomorphism of functors
\[
	f^*\Omega_\psi 
		\simeq \Omega_\phi g^!
			\colon \QQ^\antishriek\hy\Cog 
				\longrightarrow
					\PP\hy\Alg.
\]
\end{cor}

\begin{proof}
Our hypothesis on $g^!$ implies it preserves 
quasi-isomorphisms~\cite{dPBW}*{Corollary 1.1}.
This implies that there are quasi-isomorphisms 
of functors
\[
\Omega_\phi g^!  
	\stackrel{\sim}{\longrightarrow} 
		\Omega_\phi g^!\Bar_\psi\Omega_\psi 
			= \Omega_\phi \Bar_\phi f^*\Omega_\psi 
		\stackrel{\sim}{\longrightarrow}
	f^*\Omega_\psi.
\]
This is what we wanted.
\end{proof}

Recall from the discussion after Proposition~\ref{par:kahler} 
that the operadic $A$-module $\Omega^1_A$ of K\"ahler differentials 
on $A$ is the quotient of the free operadic
$A$-module generated by symbols $da$ for $a\in A$ subject to 
a generalized Leibniz rule.

\begin{lemma}\label{lema:cobars} The
map $\HKR_A$ of Proposition~\ref{prop:natHKR} is a quasi-isomorphism for $\QQ$-algebras
obtained as a cobar construction.
\end{lemma}
 
\begin{proof}
For a $\QQ$-algebra of the form $A=
\Omega_{\psi}(V)$, with $V$ a $\QQ^\antishriek$-coalgebra, 
let us compute the quasi-isomorphism type of $\Def_*(f^*A)$. 
Taking the cofibrant resolution $Q =\Omega_\phi\Bar_\phi f^* A$
of $A$ in the category of $\PP$-algebras given by the 
bar-cobar resolution, we have that:
\begin{align*}
\Def_*(f^*A) &= (\Omega^1_{Q} \otimes_{UQ} Q,d)   
	& \text{by Definition~\ref{def:cotangent}}  \\
	&= (\Bar_\phi f^*A\otimes Q, d^t)   
	&\text{by Lemma~\ref{lema:isos}}  \\
	&= \left(f^*A\otimes g^!\Bar_\psi\Omega_\psi(V), d^t\right)   	
	&\text{by Lemma~\ref{lem:fundamental lemma}}   \\
	&\hspace{-0.3 em}\stackrel{\sim}\leftarrow
		\left(f^*A   \otimes g^!(V), d^t\right)    
	&\text{since $g^!$ is PBW}. 
\end{align*}
Here $d^t$ denotes the only non-internal differential which
is transported along the isomorphism of graded vector spaces 
provided by Lemma~\ref{lema:isos}. 

On the other hand, since $A$ is quasi-free, $\Omega^1_A 
= \left(U_\QQ(A) \otimes V, d^t\right)$, endowed with an
additional transferred differential. It follows that as 
an operadic $A$-module, 
\[ 
	\Omega^*_{\FF,A} 
		=\left(A\otimes g^!(V), d^t\right),									\] 
so in particular as chain complexes $\Omega^*_{\FF,A}$
and $\Der_*(f^*A)$ are quasi-isomorphic.
It remains to see that this quasi-isomorphism gives 
an quasi-inverse to $\HKR_A$. Keeping track of the 
quasi-isomorphisms above, one can see that filtering 
$\Def_*(f^*A)$ by the appropriate word lengths
we recover the quasi-inverse at the level of the 
associated graded complexes, which reduces our
claim to that of free $\QQ$-algebras (with zero differential).
This is what we wanted. 
\end{proof}

\begin{prop}\label{prop:forms}
Suppose that $A$ is a smooth $\QQ$-algebra and
that the functor $g^!$ preserves quasi-isomorphisms.
Then $\Omega_{\FF,A}^*$ is a complex with 
homology $g^!_A\HH_0(A,A)$ concentrated in degree zero. 
\end{prop}

\begin{proof}
This follows immediately since $\Def_*(A,A) = 
Q\otimes_U \Omega_Q^1$ is a complex with homology 
concentrated in degree $0$, where it equals
$	\HH_0(A,A) = 
		A\otimes_U \Omega_A^1$,
while $g_A^!\HH_0(A,A) = 
 A\otimes g^!(\Omega_A^1)$, 
so all we need to conclude is the fact 
$g^!$ preserves quasi-isomorphisms.
\end{proof}

\begin{proof}[Proof of Theorem \ref{thm:morphism}]
Since the algebra $A$ is smooth, we know by Proposition~\ref{prop:reg} that the morphism
$p_!\Omega_Q^1 
		\longrightarrow 
			\Omega_A^1 $
is a quasi-isomorphism, and hence so is the map
\[ 
	Q\otimes_{UQ}\Omega_Q^1 
		\longrightarrow A
			\otimes_{UA}\Omega_A^1.
									\] 
By Proposition~\ref{prop:forms}, the resulting map 
$\Omega^*_{\FF,\cof } \to \Omega^*_{\FF,A}$ is a quasi-isomorphism. We have shown in Lemma~\ref{lema:cobars} that $\HKR_{\cof A}$
is a quasi-isomorphism (for $\cof A$ is cofibrant) and 
since $A$ is smooth,  the commutativity of the following diagram
$$
\begin{tikzcd}[column sep = large]
\Def_*(f^*\cof ,f^*\cof ) 
	\arrow[r,"\HKR_{\cof }","\sim"']\arrow[d,"\sim"]
		& \Omega^*_{\FF,\cof }\arrow[d,"\sim"] \\
			\Def_*(f^*A,f^*A) \arrow[r,"\HKR_A",swap]
					& \Omega^*_{\FF,A}
						\end{tikzcd}
								$$	
shows that $\HKR_{A}$ is a quasi-isomorphism.  
This is what we wanted.
\end{proof}
\definecolor{clay}{rgb}{0.8, 0.33, 0.0}

\subsection{The case of cohomology}

As before, let us take a quasi-free resolution 
$(\PP\circ\YY,d)$ of $\QQ$ where $\YY$ is a diagonal 
endofunctor and, for the $\QQ$-algebra $A$, let 
$(\QQ(V),d)$ be a quasi-free resolution.
Recall that in this case the underlying graded
vector space to $\Def_\QQ^*(A)$ is given by 
$\hom(V,\QQ(V))$.

\begin{definition}\label{def:multivector fields}
Let $A$ be a $\QQ$-algebra. We define $\Tpoly{A}$, 
the \new{poly-vector fields on $A$ relative to~$f$}, 
to be the chain complex \[\Tpoly{A} \coloneqq 
(\hom(\YY(V),\QQ(V)),\delta).\]
\end{definition}

The dual result for tangent cohomology of smooth
$\QQ$-algebras is the following. The proof is quite
similar to the case of cotangent homology, so we 
only sketch the details. We will make use of the
`smaller complex' $\Tpoly{Q,A} \coloneqq \hom(\YY(V),A)$. 
The non-internal part of the differential makes use of 
the map $(\QQ(V),d)\to A$.

\begin{theorem}\label{thm:main cohomology}
	If $f$ is left Koszul then for every smooth 
	$\QQ$-algebra $A$ the map
	\[ 
		\HKR^A :  
			\Def^*(f^*A) 
				\longrightarrow 
							\Tpoly A  
										\] 
is a quasi-isomorphism.
\end{theorem}

\begin{lemma} 
The	conclusion of Theorem~\ref{thm:main cohomology} 
holds for cobar algebras.
\end{lemma}

\begin{proof}
	
	For a $\QQ$-algebra of the form 
	$A=\Omega_{\psi}(V)$, with $V$ a $\QQ^\antishriek$-coalgebra, 
	let us compute the quasi-isomorphism type of $\Der^*(f^*A)$. 
	Taking the cofibrant resolution $Q$ given by the bar-cobar 
	resolution $\Omega_\phi\Bar_\phi f^* A$
	of $f^*A$ in the category of $\PP$-algebras, we have that:
	\begin{align*}
	\Def^*_\PP(f^*A) &= 
	\hom_{Q}(\Omega^1_{UQ}, Q)  
		&\text{by Definition~\ref{def:tangent}}  \\
	&= \hom(\Bar_\phi f^*A,Q)
	&\text{by Lemma \ref{lema:isos}}  \\
	&= \hom( g^! \Bar_\psi \Omega_\psi(V), Q) 
	&\text{by Lemma~\ref{lem:fundamental lemma}}   \\
	&\stackrel{\sim}{\to} \hom(g^!(V), Q)	
	&\text{$g^!$ is PBW}. \\
	& = \hom( \YY(V), Q) 
	\end{align*}
	Proceeding as in the case of homology, we obtain a quasi-inverse
	to the HKR map.
	\end{proof}

\begin{proof}[Proof of Theorem \ref{thm:main cohomology}]
	If $A$ is a smooth algebra, then for any cofibrant replacement $p\colon Q\stackrel{\sim}{\longrightarrow} A$ there are induced quasi-isomorphisms
	\[\Der(Q,Q) \longrightarrow \Der(Q,A)\longleftarrow \Der(A,A)\]
since $\hom_{UQ}(\Omega^1_Q,-)$ is exact.
Furthermore, similarly to Proposition \ref{prop:reg} one can show that  $\Tpoly{Q}\to \Tpoly{Q,A}$ and $\Tpoly{A} \to \Tpoly{Q,A}$ are quasi-isomorphisms.
The result follows from the commutativity of the diagram
	
	$$
	\begin{tikzcd}[column sep = large]
	\Def^*(f^*\cof ,f^*\cof ) 
	\arrow[r,"\HKR^{\cof }","\sim"']\arrow[d,"\sim"]
	& \Tpoly \cof \arrow[d,"\sim"] \\
	\Def^*(f^*Q,f^*A) \arrow[r]
	& \Tpoly{Q,A}
	\\
	\Def^*(f^*A,f^*A) \arrow[u,"\cong",swap]\arrow[r,"\HKR^A",swap]
	& \Tpoly A\arrow[u,"\sim",swap]
	\end{tikzcd}
	$$	
	where we now use coefficients to be able to draw the
	zig-zag of quasi-isomorphisms.
\end{proof}

\begin{remark}
The classical cohomological version of the HKR theorem establishes 
not only that the Hochschild cohomology and the space
of poly-vector fields of a smooth commutative algebra are isomorphic 
as chain complexes, but that they are isomorphic Lie algebras. 
In the operadic setting, tangent cohomology is also a Lie algebra via 
the bracket defined by the commutator of derivations. 
Unless the endofunctor $\YY$ carries some extra structure, it 
is not clear a priori how to endow $\Tpoly{A}$ with a Lie algebra
structure that makes our map an isomorphism of Lie algebras.
However, we point the reader to Theorem~\ref{thm:filtered} below
where $\YY$ can be taken to be an operad itself, and where we 
explain how in the classical case we, \emph{do} recover the
Lie algebra structure on poly-vector fields.
\end{remark}

\pagebreak
\section{Computations and examples}\label{sec:examples}

\subsection{Revisiting the classical HKR theorem}

Let us show that our formalism
recovers the classical HKR theorem
exactly.

\begin{prop}
The morphism $\A\longrightarrow \C$ enjoys
the HKR property, and the induced map
\[ \HKR_A : \Def_*(f^*A) \longrightarrow 
\Omega_A^* \] 
coincides, up to a suspension, with
the classical HKR quasi-isomorphism.
\end{prop}

\begin{proof}
We offer two points of view:
\begin{tenumerate}
\item We can produce a resolution of the form 
$(\A\circ \L^\antishriek,d)$ coming from the 
functorial Koszul resolution on free commutative
algebras
\[ g_V : (T(S^c(V[-1])[1]),d) \longrightarrow S(V), \]
which is manifestly diagonally pure, since the homological 
degree in $S^c(V[-1])[1]= \Lie^\antishriek(V)$ 
coincides with the weight degree.

\item The Koszul dual morphism is PBW, that is, 
$\A$ is a free right $\Lie$-module, so 
Theorem~\ref{thm:duality} implies the result.
Moreover, the lemma preceding it shows that we
may take $\YY=\Lie^\antishriek$ as in the 
previous item; we need only pay attention to 
the shift in homological degree. 
\end{tenumerate}

We conclude, in particular, that $\Tor^\A(\kk,\C)
\simeq \Lie^\antishriek$ as weight graded dg 
$\Sigma$-modules, so we may take 
$V\longmapsto \YY(V) = S^c(V[-1])[1]$
as the functor witnessing the classical HKR
property. The  HKR map for a commutative algebra~$A$ induces an isomorphism
\[
s^{-1}\HHC_*(A,A)^+=\HH_*(A,A) \longrightarrow
 \Omega_A^* =S_A^c(\Omega_A^1[-1])[1]. 
\]
It differs from the classical HKR isomorphism
map by a desuspension and in that we 
do not have, in degree zero, the identity map
of $A=\HHC_0(A)$ onto $A$. Otherwise, our 
formalism recovers the HKR map exactly.
\end{proof}
 
\subsection{New HKR theorems}\label{sec:new HKR theorems}

\textit{{Permutative algebras.}} 
A permutative algebra~\cite{Chapo} is an associative algebra 
$A$ such that for every $x,y$ and $z\in A$, we have that
$x(yz) = x(zy)$. Permutative algebras are algebras over 
a binary quadratic operad, denoted $\Perm$ which is the
linearisation of a set operad $\Perm_0.$
The Koszul dual operad of $\Perm$ is the operad $\preLie$
controlling pre-Lie algebras. Both these algebraic structures
play an important role in the study of operadic deformation 
theory~\cite{calaque2019moduli, TwistingProcedure,
PreLieDeformation}.

Clearly, every commutative algebra is a permutative algebra
via the same product. Since a permutative product is 
in particular associative, there is a factorisation of 
the map of operads $f\colon \A \longrightarrow\C$ via 
the permutative operad: 
$
 \A \longrightarrow 
 	\Perm \stackrel{\psi}{\longrightarrow} \C.
$

\begin{prop}
The map $\psi \colon \Perm \to \C$ enjoys the HKR property, with
generating sequence $\mathsf{RT}_{\neq 1}$ that assigns a set
$I$ to the set of rooted trees with vertices labeled by $I$
for which no vertex has exactly one child.
\end{prop}

\begin{proof}
The Koszul dual map to $\psi$ is the anti-symmetrisation map 
$\phi \colon \Lie \to \preLie$. In \cite{PBW} this map was 
shown to satisfy the PBW property. Moreover,
in~\cite{dotsenko2020schur}, it was show that the generators of 
$\preLie$ as a right $\Lie$-module are given by the functor
$\mathsf{RT}_{\neq 1}$ as in the statement of the proposition. 
The result follows from Corollary~\ref{cor:PBW}.
\end{proof}

Theorems \ref{thm:morphism} and \ref{thm:main cohomology}
allow us to compute the permutative (co)tangent 
(co)homology of commutative algebras.

\begin{cor}\label{cor:perm}
Let $A$ be a smooth commutative algebra. 

\begin{titemize}
\item 
The cotangent homology $\HH_*(\psi^*A)$  is isomorphic to the algebra of ``tree-wise'' 
differential forms $\mathsf{RT}_{\neq 1}(\Omega_A^*)$ over the classical differential
forms of $A$.

\item 
Dually, the tangent cohomology $\HH^*(\psi^*A)$ is 
isomorphic to the algebra of ``tree-wise'' poly-vector 
fields $\mathsf{RT}_{\neq 1}^\vee(\Tpoly A)$.\qed
\end{titemize}
\end{cor}

\begin{cor}
For every smooth commutative algebra $A$ there is a 
quasi-isomorphism
\[ 
\Def_{\Perm}^*(A) \longrightarrow 
\mathsf{RT}_{\neq 1}^\vee(\Tpoly A).
\]
Moreover, the natural map $\Def_{\Perm}^*(A) \longrightarrow	
\Def_{\A}^*(A) $ induces, in homology, the natural map 
\[ 
\mathsf{RT}_{\neq 1}^\vee(\Tpoly A)
 \longrightarrow 
 \Tpoly A
 \]
given by the augmentation 
$\mathsf{RT}_{\neq 1}^\vee\longrightarrow \kk$.\qed
\end{cor} 

\medskip
\textit{{Enriched pre-Lie algebras of Dotsenko and Foissy.}} 
In~\cite{DotsenkoFoissy} the authors define a functor  that
assigns to every Hopf cooperad $\mathcal{C}$ an operad 
$\preLie_\mathcal{C}$ of $\mathcal{C}$-enriched pre-Lie
algebras. The example we are interested in is the following,
where this functor recovers the operad 
of pre-Lie algebras and the operad of braces algebras. Such
braces algebras and related structures, 
conceived originally in~\cite{BracesTornike},
appeared in the literature
in several opportunities~\cite{GetzlerGaussManin,Gerstenhaber1995,Lada2005}, and
are relevant in deformation theory~\cite{PreLieDeformation},
for example. 

\begin{tenumerate}
\item 
If $\mathcal{C}=u\C^*$ is the Hopf cooperad of \emph{unital} 
commutative coalgebras, one obtains the operad $\preLie$
governing pre-Lie algebras.
\item
If $\mathcal{C}=u\A^*$ is the Hopf cooperad of \emph{unital} 
associative coalgebras, one obtains the operad $\Br$
governing brace algebras.
\item The unit map $u\C^*\longrightarrow u\A^*$ gives the map
$g:\preLie\longrightarrow \Br$ constructed in~\cite{DailyLada}.
\end{tenumerate}

By Proposition 2 in that article, every connected Hopf
cooperad  $\mathcal{C}$ admits a structure of associative
algebra for the Cauchy product in the category symmetric
sequences ---what are usually called twisted associative 
algebras---, in such a way that every morphism of Hopf
cooperads $\mathcal{C}\longrightarrow \mathcal{C}'$
induces a morphism of twisted associative algebras. 
\begin{definition}[Proof of Theorem 1 in~\cite{DotsenkoFoissy}]
Given a species $\mathcal X$, there is a species of enriched
trees, which we write $\mathcal T_R$, where
each vertex is decorated by an element of $\mathcal{C}'$
with the condition that every vertex of maximal depth is 
decorated by an element of $\mathcal X$. Similarly,
$\mathcal T_L$ is the species of 
$\mathcal{C}'$-enriched trees, with the condition that
the root vertex is decorated by an element of $\mathcal X$.
\end{definition}
The main result of~\cite{DotsenkoFoissy} is as follows.

\begin{theorem*} Let $\varphi : \mathcal{C}\longrightarrow 
\mathcal{C}'$ be a morphism of connected Hopf cooperads and
let us consider $\mathcal{C}'$ as a $\mathcal{C}$-bimodule 
by viewing $\varphi$ as a map of twisted associative algebras.  
Then:
\begin{tenumerate}
\item if $\mathcal{C}'$ is left $\mathcal{C}$-free with 
generators $\mathcal X$ then the operad $\preLie_{\mathcal{C}'}$
is free as a left $\preLie_\mathcal{C}$-module with generators
$\mathcal T_L$ and,
\item if $\mathcal{C}'$ is right $\mathcal{C}$-free with 
generators $\mathcal X$ then the operad $\preLie_{\mathcal{C}'}$
is free as a right $\preLie_\mathcal{C}$-module with generators
$\mathcal T_{R}$.\hfill\qed
\end{tenumerate}
\end{theorem*}  

An immediate consequence of this result is the following, since the map of twisted
associative algebras $u\C^*\longrightarrow u\A^*$ is both
left and right free.

\begin{cor}[Theorem 2 in~\cite{DotsenkoFoissy}]\label{cor:BrLie}
The brace operad $\Br$ is free as a left and as 
a right $\preLie$-module.\qed
\end{cor}

From this, we obtain the following HKR theorem for smooth
brace algebras.

\begin{theorem}\label{thm:BrPLie}
The map $g:\preLie\longrightarrow \Br$ satisfies the HKR
property: for every smooth brace algebra $A$ there exists
a natural quasi-isomorphism
\[ \HKR_A : \Def_*(g^*A) \longrightarrow \Omega_{\FF,A}^* \]
where $\Omega_{\FF,A}^* = \mathcal T_R(\Omega_A^1)$ where
$\mathcal T_R$ is the endofunctor of rooted
trees with vertices of maximal depth decorated by Lie words.\qed
\end{theorem}

\medskip

\textit{Diassociative algebras.} Diassociative algebras 
were introduced by J.-L. Loday in~\cite{LodayDialgebras}. 
A diassociative algebra~\cite[Section 13.6]{LV} consists
of a vector space $V$ along with two associative operations 
\[\vdash \colon V\otimes V \longrightarrow V
	\quad\text{and}\quad
	\dashv \colon V\otimes V \longrightarrow V\] 
satisfying the following set of three quadratic relations:
\begin{equation*}
(x_1\dashv x_2)\dashv x_3 = x_1 \dashv (x_2\vdash x_3), 
\quad (x_1\vdash x_2)\dashv x_3 = x_1 \vdash (x_2\dashv x_3), 
\quad (x_1\dashv x_2)\vdash x_3 = x_1 \vdash (x_2\vdash x_3).
\end{equation*}
Any permutative algebra gives rise to a diassociative algebra 
by defining both products to be the permutative product, so
that we have a map $\mathsf{Dias}\longrightarrow \mathsf{Perm}$,
whose Koszul dual is the map $\mathsf{PreLie}\longrightarrow
\mathsf{Dend}$. In~\cite{PBW} the authors proved that this
morphism is PBW and, since it is known that $\mathsf{Dend}$
is a \emph{left} free $\A$-module with basis the operad of braces
$\mathsf{Br}$, we can use Corollary~\ref{cor:BrLie} and our
main theorem to obtain the following result:

\begin{theorem}
For every smooth permutative algebra $A$ there is a 
quasi-isomorphism
\[ \HKR_A : \Def_*^{\mathsf{Dias}}(A) 
	\longrightarrow \Omega_A^* \]
where the endofunctor $\YY$ is given by $\A\circ\mathcal T_R$
and $\mathcal T_R$ is the endofunctor of Theorem~\ref{thm:BrPLie}.
\qed
\end{theorem}

\medskip

\textit{The work of J. Griffin.} 
Let us now connect our formalism with the one developed by 
J.~Griffin. Motivated by the Hodge decomposition of 
Hochschild cohomology~\cite{BarrHodge,Gerstenhaber1987}, 
Griffin~\cite{Griffin} considered the problem ---like we 
do--- of computing the cohomology of a pull-back algebra 
$f^*A$ under a morphism of operads $f:\PP\longrightarrow
\QQ$. Since his motivation is slightly different from ours,
there is no mention of HKR-type theorems in his paper, nor 
of smooth algebras. 

However, one can find the following result in \textit{ibidem}, which
relates the Quillen homology of a $\QQ$-algebra $A$ to that
of its pull-back, which can be seen as a first approximation
to the problem of computing the cohomology of $f^*A$, and which
contains already a clear link between his and our formalism; 
see~Theorem 3.7-(II) in~\cite{Griffin}.

\begin{theorem}
Let $f:\PP\longrightarrow \QQ$ be a map of Koszul operads and let
$g:\QQ^\antishriek \longrightarrow \PP^\antishriek$ be its Koszul
dual map. Suppose that $\PP^\antishriek$ is a free right 
$\QQ^\antishriek$-comodule with basis~$\mathcal{X}$. 
Then for every $\QQ$-algebra $A$ there is an isomorphism
\[ 
\Bar_\PP(f^*A) 
\longrightarrow 
\mathcal{X}\circ\Bar_\QQ(A).
\hfill\qed
\]
\end{theorem}

Moreover, Griffin goes on to consider the case of maps of
Koszul operads $\PP\longrightarrow \QQ$ where $\QQ$ is obtained
from $\PP'$ and another operad $\PP$ by a filtered distributive
law~\cite{FilteredLaw} as originally defined by V. Dotsenko 
in~\cite{Filtered}; see~\cite{Griffin}*{Theorems 5.15 and 5.18}. 
The following result offers a complementary technique to compute
the tangent cohomology of a $\PP$-algebra coming from 
a smooth $\QQ$-algebra under the projection 
$f:\PP\longrightarrow \QQ$.

\begin{theorem}[Filtered HKR theorem]\label{thm:filtered}
Suppose $\PP$ is obtained from $\QQ$ and $\mathcal R$ by a filtered
distributive law, so that $\PP$ is isomorphic to $\QQ\circ\mathcal R$
as a right $\mathcal R$-module. For every smooth $\QQ$-algebra $A$
the cotangent homology of $f^*A$ is given by the endofunctor
\[ A\longrightarrow 	\mathcal R^{\emph\antishriek}(\Omega_A^1)\]
of ``$\mathcal R^{\emph\antishriek}$-enriched differential forms'' on $A$.
\end{theorem}

\begin{proof}
Since we are working over a field of characteristic zero,
Theorem 5.4 in~\cite{FilteredLaw} guarantees that $\PP^!$
is a free right $\QQ^!$-module with generators $\mathcal R^!$,
so the claim follows.
\end{proof}

\subsection{The operadic butterfly of J.-L. Loday}

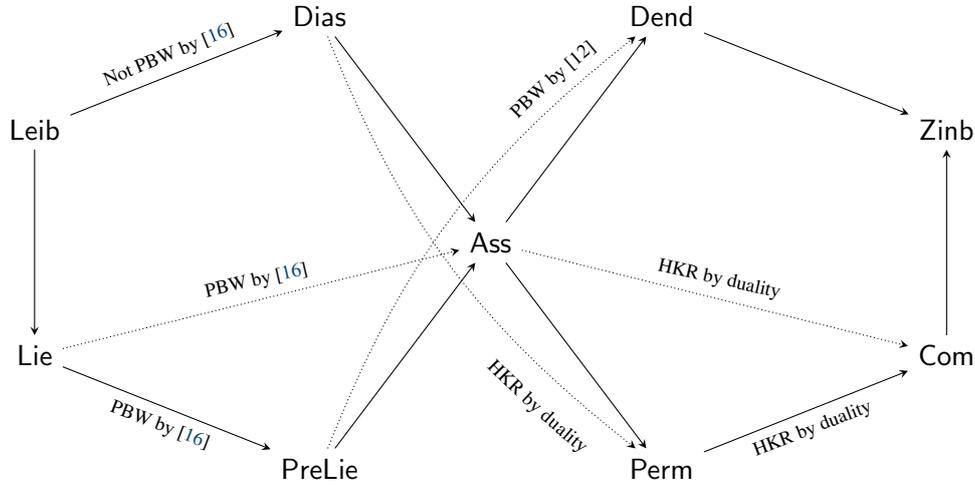
\begin{figure}
\begin{center}
 \begin{tikzpicture}[scale = 0.75]
\tikzset{->,>=stealth}
    		\node (A) at (1, 0) {$\mathsf{Ass}$};
    		\node (P) at (4, -4) {$\mathsf{Perm}$};
    		\node (C) at (9, -2) {$\mathsf{Com}$};
    		\node (Z) at (9, 2) {$\mathsf{Zinb}$};
    		\node (D) at (4, 4) {$\mathsf{Dend}$};
    		\node (D') at (-2, 4) {$\mathsf{Dias}$};
    		\node (PL) at (-2, -4) {$\mathsf{PreLie}$};
    		\node (L) at (-7, -2) {$\mathsf{Lie}$};
    		\node (L') at (-7,2 ) {$\mathsf{Leib}$};
    		\draw (A) to (P);
    		\draw[densely dotted] (L) to node[above,rotate= 13,black]{\scriptsize PBW by~\cite{PBW}}(A);
    		\draw (P) to node[below,rotate= 21,black]{\scriptsize HKR by duality}(C);
    		\draw (C) to node[above,rotate= -90,black] {} (Z);
    		\draw (D) to (Z);
    		\draw (A) to (D);
    			\draw[densely dotted] (A) to node[above,rotate= -13,black]{\scriptsize HKR by duality}
    (C);
    		\draw[bend right =15, densely dotted,
    			postaction={decorate,
    					decoration={text along path,
    					text align={right,right indent=0.5 cm}, 		
    					text={|\scriptsize| HKR by duality},
    					raise=- 5 mm,
    					%text effects/.cd,
    					%characters={text along path, yshift=-5pt}
    						}
    					}
    							] (D') to (P);
    		
    		\draw (PL) to (A);
    		\draw [bend left =15, densely dotted] (PL) to (D);
    		\path [bend left =15,
    			postaction={decorate,
    				decoration={text along path,
    							text align={right,right indent=0.5 cm}, 
    							text={|\scriptsize|PBW by [12]},
    							raise=2 mm	
    							}
    						}
    					] (PL) to (D);
    		\draw (L) to node[below,rotate= -21,black]{\scriptsize PBW by~\cite{PBW}} (PL);
    		\draw (D') to (A);
    		\draw (L') to node[above,rotate= 22.5,black]{\scriptsize Not PBW by~\cite{PBW}}(D');
    		\draw (L') to (L);
    \end{tikzpicture}
\end{center}
\caption{The operadic butterfly.}
\end{figure}

Let us recall from~\cite{Butterfly} 
that we can arrange certain nine operads into a ``butterfly' diagram of morphisms, as in the figure above. We
record those maps which we know satisfy the PBW property
and which we know satisfy the HKR property. Most of the claims follow
immediately by duality (Theorem~\ref{thm:duality}) from the results obtained in~\cite{PBW},
or by the following simple remark:

\begin{remark}\label{rem:size}
Note that if $f:\PP\longrightarrow \QQ$ is PBW, then
we must have $\dim_\kk\PP(n) \leqslant\dim_\kk\QQ(n)$
for each $n\in\NN$. In particular, since $\dim_\kk\mathsf{Dias}(2) > \dim_\kk\mathsf{Ass}(2)$, $\dim_\kk\mathsf{Leib}(2) > \dim_\kk\mathsf{Lie}(2)$, and since $\dim_\kk\mathsf{PreLie}(3) > \dim_\kk\mathsf{Ass}(3)$, the respective maps in Figure 1 are not PBW.
\end{remark}

It would be interesting to determine if the remaining arrows
enjoy the HKR or the PBW property or if, perhaps, they enjoy
none of the two.

\begin{remark} It is well
known~\cite{Chapoton2010} that the map of operads
$\mathsf{Lie}\longrightarrow 
\mathsf{PreLie}$ makes its 
codomain a free left module. 
However, the generators
exhibiting $\mathsf{PreLie}$ as
a left free $\mathsf{Lie}$-module
are not concentrated in weight
zero, so that, as expected,
the map from $\mathsf{Perm}$
onto $\mathsf{Com}$ is not PBW. In fact, in general,
the extra weight degree we have considered means
a map $f:\PP\longrightarrow \QQ$ that is left free
will not be left Koszul unless it is the identity,
which shows that it is crucial to replace the `left free'
condition to a left Koszul condition. 
\end{remark}

\appendix

\section{Recollections on operads}\label{sec:recollections}
\subsection{Operads and their algebras and modules}

Let us fix a reduced symmetric operad $\mathcal P$ and write 
$\mathcal P\hy\Alg$ for the category of dg $\PP$-algebras. 
The operad $\PP$, viewed as a monad, gives the left adjoint 
\[\PP: \SMod \longrightarrow \PP\hy\Alg \] to the forgetful
functor \[\# : \PP\hy\Alg\longrightarrow \SMod.\]

Fix a dg $\PP$-algebra $A$ as before. An \new{operadic $A$-module} is
a dg $\Sigma$-module $M$ along with an action 
$\gamma_M :
	\PP\circ (A,M)
		\longrightarrow M$ 
so that 
\[\gamma_M 
	(1\circ(\gamma_A,\gamma_M)) 
		= \gamma_M(\gamma\circ (1,1)).\] 
Here $\PP\circ (A,M)$ is the submodule of $\PP(A\oplus M)$ which
is linear in $M$. 
	
It is useful to note that if $\PP=\mathsf{As}$ and if $A$ is an 
$\PP$-algebra or, what is the same, an associative algebra, then
an operadic $A$-module is the same as an $A$-bimodule and \emph{not}
a left (or right) $A$-module. Similarly, the operadic modules for 
commutative algebras are the symmetric bimodules. In fact, there is a functor 
 \[U_\PP : 
 	\PP\hy\Alg	
 			\longrightarrow
					\A\hy\Alg,\] 
the last being the category of dga algebras, so that the category
of operadic $A$-modules is isomorphic to the category of left 
$U_\PP(A)$-modules of the associative algebra $U_\PP(A)$. 

\begin{definition} We call $U_\PP(A)$ the \new{associative enveloping algebra of $A$}. \end{definition}

Concretely, $U_\PP(A)$ is spanned by trees
with one leaf pointed by the only element in~$\kk$ under
 the relation that identifies the corolla with root 
 $\mu\circ_i \nu$ with the corolla with root  
$\mu$ and $\nu$ acting on the leaves $i,i+1,\ldots$, and 
we will write a generic element by 
\[ u(a_1,\ldots,a_{i-1},-,a_{i+1},\ldots,a_n)\]
where $u$ is an operation of $\PP$ and the empty slot
corresponds to the leaf marked by~$\kk$. The algebra
structure is defined by concatenation through the pointed leaf
and the root through the partial composition $\circ_i$ of $\PP$. 
We refer the reader to~\cite{Anton} for a useful reinterpretation
of $U_\PP$ through the language of 2-colored operads.

As useful examples, we note that in the case 
of associative and Lie algebras, we recover
the usual notion of enveloping algebra: for an 
associative algebra $A$ we have that 
$U_{\mathsf As}(A) = A\otimes A^{\rm op}$, 
for a Lie algebra $L$ we have that 
$U_{\mathsf{Lie}}(L) = U(L)$; note in 
both cases we are considering non-unital 
algebras and non-unital actions.

Given a map of $\PP$-algebras $f:B\longrightarrow A$, we obtain
two maps 
\[ f^*:{}_A\Mod
		\longrightarrow {}_B\Mod 
			\,\text { and } 
 \,f_{\,!} : {}_B\Mod
 		 \longrightarrow
						{}_A\Mod\] 
corresponding respectively to the restriction and extension
of scalars,  and a map 
\[ U_\PP(f) : U_\PP(B)\longrightarrow U_\PP(A).\]
Then the previous two adjoint functors are simply the usual
functors of restriction and extension for $U_\PP(f)$. 
We can also describe the free modules
as follows. If $X$ is a dg $\Sigma$-module, we have a coequalizer
diagram 
\[
	\begin{tikzcd}
 		\PP(\PP(A),X) \arrow[r,shift left = .25 em]
 				\arrow[r,swap,shift right= .25 em] & 
 				\PP(A,X) \arrow[r] & A \circ_\PP X 
 				\end{tikzcd}
\] 
where the arrows are $\gamma(1,1)$ and $1(\gamma_A,1)$ and $A \circ_\PP X$
is the free operadic $A$-module on $X$, so that 
\[ A\circ_\PP - : \SMod
\longrightarrow {}_A\Mod
\]
 is left adjoint to the forgetful 
functor 
\[ \# : {}_A\Mod\longrightarrow \SMod.\]
Graphically, generators of $A\circ_\PP X$ correspond
to corollas with their root labeled by an operation
of $P$, all whose leaves are labeled by elements of
$A$ except for one, which is labeled by an element
of $X$, and we impose the relations for each $i,l,n\in\NN$, each pair of operations $\mu,\nu\in \PP$ with $\mu$ of arity $n$ and each $n$-tuple $(a_1,\ldots,a_n)$ of elements of $A$,
\[ \mu(a_1,\ldots,a_i,\nu(a_{i+1},\ldots,a_{i+l}),a_{i+l+1},\ldots, a_n,x) =
	(\mu\circ_i\nu)(a_1,\ldots,a_n,x).\] 
In case $A$ or $\PP$ are graded, signs will
appear owing to the Koszul sign rule. 	 

\subsection{Derivations and Kähler differentials}\label{app:derivations}

As before, let us fix an operad
$\PP$, and let us also fix a $\PP$-algebra $A$.  If $M$ is an operadic $A$-module then a \new{$\PP$-derivation of $M$} is a linear map $d : A\longrightarrow M$ such that $\gamma_M(1\circ' d) = d\gamma_A$. For a fixed choice $f:B\longrightarrow A$
of a map $\PP$-algebras, we say $d$ is $B$-linear whenever it vanishes on the image of $f$. Following~\cite{hinich1997homological}, we write $\Der_B(A,M)$ for the complex
of such derivations, which defines a functor 
\[
\Der_B(A,-)\colon {}_A\Mod\longrightarrow \Cxs.\]
In particular, if $u
\colon A\longrightarrow U$ is a map of 
$\PP$-algebras, then $U$ is
an operadic $A$-module and we can
consider $\Der_B(A,U)$
the complex of $B$-linear derivations $A\longrightarrow U$. We refer
the reader to~\cite{LV}*{\S 12.3.19}
for a proof of the following:

\begin{prop}\label{par:kahler} The functor $\Der_B(A,-)$ is representable. \qed
\end{prop} 
We call the representing module the module of \new{relative Kähler differentials}
and write it $\Omega_{A\vert B}^1$. Explicitly, $\Omega_A^1$ is 
the coequalizer of the diagram 
\[\begin{tikzcd}
 		 A\circ_\PP d \PP(A)
 		 \arrow[r,shift left = .25 em]
 				\arrow[r,swap,shift right= .25 em] & 
 				A\circ_\PP dA \arrow[r] & \Omega_A^1
 				\end{tikzcd}
\]
so that $ \Omega_A^1$ is the free operadic
$A$-module on a copy
$dA$ of $A$ where we additionally impose the relations that, for 
each $i,l,n\in\NN$, each pair of operations $\mu,\nu\in \PP$ with 
$\mu$ of arity $n$ and each $n$-tuple $(a_1,\ldots,a_n)$ of 
elements of $A$, where $a'=d\nu(a_{i+1},\ldots,a_{i+l})$: 
\[ \mu(a_1,\ldots,a_i,a',a_{i+l+1},\ldots, a_n) =
	\sum_{t=1}^l (\mu\circ_i\nu)(a_1,\ldots,da_{i+t},\ldots,a_n)\]

The arrows are as follows: 
the uppermost arrow is induced from the map 
\[ 1(1,\gamma_A): \PP(A,d\PP(A))\longrightarrow \PP(A,dA),\]
while the lowermost arrow is induced from the 
following three maps:
\begin{tenumerate}
\item the arrow $\PP(A,d\PP(A))\longrightarrow \PP(A,\PP(A,dA))$ 
induced from the infinitesimal composite $1\circ' d : d\PP(A)
 \longrightarrow \PP(A,dA)$ obtained from the isomorphism $d:A\longrightarrow dA$, 
\item the arrow $\PP(A,\PP(A,dA))\longrightarrow (\PP\circ_{(1)} \PP)(A,dA)$ which is 
an inclusion and 
\item the arrow $\gamma_{(1)}(1,1)$.
\end{tenumerate}
The module of relative K\"ahler differentials $\Omega_{A\mid B}^1$ is defined similarly, with 
the extra relation that $dB=0$.
It is functorial
in both arguments in the following way. If we have a pair
of morphisms $ B \xrightarrow{\hem f \hem} A\xrightarrow{\hem g\hem} C$
of $\PP$-algebras 
we can consider any $A$-linear derivation of $A$ as a
$B$-linear derivation, so we get a morphism 
\[\Omega_{C\vert g}^1 \colon \Omega_{C\vert B}^1 \longrightarrow
\Omega_{C\vert A}^1\] representing the restriction.  
Similarly, any $B$-linear derivation $d:C\longrightarrow M$ 
defines a $B$-linear derivation $g^*d : A\longrightarrow f^*M$
so we obtain a morphism 
\[\Omega_{g\vert B}^1 : g_{\,!}\Omega_{A\mid B}^1\longrightarrow 
\Omega_{C\mid B}^1.\]

{{}} The following lemma describes
K\"ahler differentials and derivations of free
algebras. In particular, it follows the 
corresponding complexes of derivations
and of differentials of $\PP$-algebras of
the form $(\PP(V),d)$ are simple,
and correspond to ``nc-vector fields''
$X:V\longrightarrow \PP(V)$ determined
on the coordinates $v\in V$ by some
vector field $X:\partial_v\longmapsto X(\partial_v)$, and to ``nc-differential forms''
$f(v) dv$ where $f(v)\in Y$ is a
function on the coordinates.

\begin{lemma}\label{lema:isos} Let $A=\PP(V)$ be
the free $\PP$-algebra on $V$. Write $i:V\longrightarrow \PP(V)$ for the
canonical inclusion. 
Then $\Omega_A^1$ is canonically isomorphic to the free operadic
$X$-module generated by $V$, and we have isomorphisms of complexes
\[ i^* : \Der(A) \longrightarrow \hom(V,A),\qquad
i_* : A\otimes V 
\longrightarrow A\otimes_{UA}\Omega_A^1\] 
that assign a derivation $f:A\longrightarrow A$ to its restriction $fi$
and $x\tt v$ to the class of $x dv$. \end{lemma}

\begin{proof}
Since $X$ is free, any derivation $f:X\longrightarrow X$ is 
determined by its restriction to $V$, and $i^*$ is a bijection. 
From this and the Yoneda lemma 
it follows that $\Omega_X^1$ is the free left $UX$-module
generated by $V$, and hence that the canonical map
$ X\otimes V \longrightarrow X\otimes_{UX} \Omega_X^1$ 
is an isomorphism. 
\end{proof}

In particular, if we consider a
commutative algebra $A$ and the bar-cobar resolution $Y = \Omega BA$, we get
the following:

\begin{lemma}\label{par:CCiso} 
There is a natural isomorphism
$ A\otimes_{UY} \Omega_Y^1 
\longrightarrow A\tt s^{-1}\overline{BA}
 = s^{-1}\overline{C_*(A,A)}$
where the right hand side is the cyclic Hochschild
complex computing Hochschild homology of $A$ away
from degree $0$. \qed
\end{lemma}

\bibliographystyle{alpha}
\begin{multicols}{2}
\bibliography{hkr-paper}
\end{multicols}
\Addresses
\end{document}